\documentclass[a4paper,10pt]{amsart}

\usepackage{graphicx}

\usepackage{mathrsfs}
\usepackage{amssymb,amsmath, amsfonts, amsthm}
\usepackage{latexsym}
\usepackage{color} 
\usepackage[dvips]{epsfig}
\usepackage{graphicx}

\newtheorem{defi}{Definition}[section]
\newtheorem{theorem}[defi]{Theorem}
\newtheorem{lemma}[defi]{Lemma}

\newtheorem{proposition}[defi]{Proposition}
\newtheorem{remark}[defi]{Remark}



\begin{document}

\author[Isabella Ianni]{Isabella Ianni}
\author[Giusi Vaira]{Giusi Vaira}
\date{}
\title[Non radial sign-changing solutions]{Non-radial sign-changing solutions for the Schr\"{o}dinger-Poisson problem in the semiclassical limit}
\address{Isabella Ianni\\  Dipartimento di Matematica e Fisica\\  Seconda Universit\`a degli Studi di Napoli\\ viale Lincoln 5, 81100 Caserta (Italy)}
\email{isabella.ianni@unina2.it}

\address{Giusi Vaira\\  Dipartimento di Scienze di Base e Applicate per l'Ingegneria, sezione di Matematica\\ Universit\`a La Sapienza di Roma\\ Via A. Scarpa 14, 00185 Roma (Italy)}
\email{giusi.vaira@sbai.uniroma1.it}

\keywords{Schr{\"o}dinger-Poisson problem, semiclassical limit, cluster solutions, sign-changing solutions, variational methods, Lyapunov-Schmidt reduction.}

\subjclass[2010]{35B40, 35J20, 35J61, 35Q40, 35Q55}

\maketitle

\begin{abstract}
We study the following system of equations known as Schr\"{o}dinger-Poisson problem
$$
\left\{
\begin{array}{lr}
-\epsilon^2\Delta v+v +\phi v=f(v)\qquad\mbox{ in }\mathbb R^N\\
-\Delta\phi =a_N v^2    \,\qquad\qquad\qquad\ \mbox{ in } \mathbb R^N\\
\phi\rightarrow 0 \ \mbox{ as }\ |x|\rightarrow +\infty
\end{array}
\right.
$$
where $\epsilon>0$ is a small parameter, $f:\mathbb R\rightarrow\mathbb R$ is given, $N\geq 3,$ $a_N$ is the surface measure of the unit sphere in $\mathbb R^N$ and the unknowns are $v, \phi:\mathbb R^N\rightarrow\mathbb R.$

We construct non-radial sign-changing multi-peak solutions in the semiclassical limit. The peaks are displaced in suitable symmetric configurations and  collapse to the same point as $\epsilon\rightarrow 0.$ The proof is based on the Lyapunov-Schmidt reduction.
\end{abstract}

\section{Introduction}
In this paper we are concerned with the existence of sign-changing solutions to the following nonlinear Schr\"{o}dinger-Poisson problem
$$
(\mathcal{SP})\qquad\qquad\qquad\qquad \left\{
\begin{array}{lr}
-\epsilon^2\Delta v+v +\phi v=f(v)\qquad\mbox{ in }\mathbb R^N\\
-\Delta\phi =a_N v^2    \qquad\qquad\qquad\,\,\,\mbox{ in } \mathbb R^N\\
\phi(x)\rightarrow 0 \ \mbox{ as }\ |x|\rightarrow +\infty
\end{array}
\right.
$$

where $\epsilon$ is a small and positive parameter,  $f:\mathbb R\rightarrow\mathbb R$ is given, $N\geq 3,$ $a_N$ is the surface measure of the unit sphere in $\mathbb R^N$ and the unknown is $(v, \phi):\mathbb R^N\times\mathbb R^N\rightarrow\mathbb R.$ \\\\
Systems like $(\mathcal{ SP})$ have been object of many investigations in the last years because of their strong physical meaning. Indeed they appear in quantum mechanics models (see e.g. \cite{5, 8, 14}) and also  in semiconductor theory \cite{BF1, BF2, 15, 16}. In \cite{BF1, BF2}, for instance, they have been introduced as  models describing solitary waves  for nonlinear stationary equations of Schr\"{o}dinger type  interacting with an electrostatic
field, and are usually known as Schr\"odinger-Poisson systems. In this context the nonlinear term $f$  simulates, as usual,  the interaction
between many particles, while the solution $\phi$ of the Poisson equation plays the role of a potential  determined by the charge of the wave function itself. From another point of view, the
interest on this problem stems also from the Slater approximation of
the exchange term in the Hartree-Fock model, see \cite{Slater}. In
this framework $f(u)=u^p$ with  $p=5/3$, however, other nonlinearities have been used
in different approximations.\\\\
In the following we look for bound states to $(\mathcal{SP})$ in the semiclassical case, namely as $\epsilon \rightarrow 0$. \\While there are many results about existence, multiplicity and behavior of positive solutions to $(\mathcal{ SP})$ (see \cite{A, AzzolliniPomponio,  DaprileMugnai1, DaprileMugnai2, DW, DaprileWei2, K,  Ruiz, R} and references therein),  little is known about the existence of solutions $(v,\phi)$ such that $v$ is sign-changing.
\\In \cite{I} the existence of solutions with $v$ nodal is established in the case $\epsilon=1,$ the solutions found are radial and $v$ has any fixed number of nodal domains. As far as we know, nothing is known about  the existence of non-radial sign-changing $v$.\\

In this paper we give an improvement in this direction. Indeed we construct, for $\epsilon$ small, non-radial solutions to $(\mathcal{SP})$ such that $v$ is nodal, moreover $v$ is multi-peak shaped and its peaks collapse all at a certain point (which we may assume to be $0$ by the invariance by translation) as $\epsilon\rightarrow 0$ (cluster nodal solutions).

We recall that D'Aprile and Wei in \cite{DW} proved  the existence of positive cluster
solutions to $(\mathcal{SP})$ as $\epsilon$ goes to zero, hence this paper completes the picture about the
existence of cluster solutions to $(\mathcal{SP}).$
\\
\\
Before stating the main results we fix the assumptions on $f$ that we will use in the sequel and we  recall some known facts.
\\
\\
(f1) $f\in C^{1+\sigma}_{loc}(\mathbb R) $ with $\sigma \in (0,1),$ $f(0)=f'(0)=0$ and $f(t)=-f(-t).$
\\
\\
(f2) the problem
\begin{equation}\label{problema limite}\left\{
\begin{array}{lr}
\Delta w -w+f(w)=0\qquad \mbox{ in }\mathbb R^N\\
w>0\,\qquad\qquad\qquad\qquad \mbox{ in }\mathbb R^N\\
\lim_{|x|\rightarrow +\infty}w(x)=0\\
w(0)=\max_{\mathbb R^N}w(x)
\end{array}\right.
\end{equation}
has a unique solution $w$ which is non degenerate, i.e. denoting by $\mathcal{L}:H^2(\mathbb R^N)\rightarrow L^2(\mathbb R^N)$ the linearized operator in $w,$
$$\mathcal{L}[u]:=\Delta u-u +f'(w)u,$$
then
$$Kernel(\mathcal{L})={\rm{span}}\left\{\frac{\partial w}{\partial x_1},\ldots,\frac{\partial w}{\partial x_N}\right\}.$$
We recall that $w$ is a critical point of the following energy functional
 $$I[w]:=\frac{1}{2}\int_{\mathbb R^N}(|\nabla w|^2+w^2)dx-\int_{\mathbb R^N} F(w)dx$$
 where $F(t)=\int_{0}^tf(s)ds.$

By the well-know result of Gidas, Ni and Nirenberg (\cite{GNN}), $w$ is radially symmetric and strictly decreasing in $r=|x|.$ Moreover, by classical regularity results, the following asymptotic behaviors hold:
\begin{equation}\label{dec}
w(r), w''(r)=A_N r^{-\frac{N-1}{2}}e^{-r}\left(1+O\left(\frac{1}{r}\right)\right),
\end{equation}
\begin{equation}
w'(r)=-A_N r^{-\frac{N-1}{2}}e^{-r}\left(1+O\left(\frac{1}{r}\right)\right),
\end{equation}
where $A_N>0$ is a suitable positive constant.

The class of nonlinearities $f$ satisfying (f1)-(f2) includes, and it's not restricted to, the homogeneous nonlinearity $f(v)=|v|^{p-1}v $ with $p\in (1, \frac{N+2}{N-2}).$
\\
In this paper the dimension $N$ is chosen in the interval $[3,6].$ Under this assumption it is well known that the system $(\mathcal{SP})$ can be reduced into a single equation. Indeed a simple application of the Lax-Milgram theorem ensures the existence of a unique solution of the second equation of $(\mathcal{SP})$, namely the following result holds:

\begin{lemma}\label{esistenzaphi} Let $N\in [3,6].$
For every $f\in L^{\frac{2N}{N+2}}(\mathbb R^N)$ there exists a unique solution $\phi[f]$ in $D^{1, 2}(\mathbb R^N)$ of the equation $-\Delta\phi=a_N f$. Moreover the following representation formula holds:
$$\phi[f](x):=\int_{\mathbb R^N} \frac{f}{|x-y|^{N-2}}dy.$$
Furthermore the functional $G:H^1(\mathbb R^N)\rightarrow \mathbb R$
$$G(u):=\int_{\mathbb R^N} \phi[u^2]u^2 dx$$ is $C^1$ and $G'(u)[v]=4\int_{\mathbb R^N}\phi[u^2]uv dx.$
\end{lemma}
By Lemma \ref{esistenzaphi} we reduce to study the following nonlinear scalar equation in $H^1(\mathbb R^N)$
\begin{equation}\label{eqScalare}-\epsilon^2\Delta v+v +v \phi[v^2] =f(v)\qquad\qquad\mbox{ in }\mathbb R^N.\end{equation}

We also recall that the solutions of \eqref{eqScalare} are critical points of the $C^2$-functional $J_{\epsilon}:H^1(\mathbb R^N)\rightarrow \mathbb R$ defined as
\[J_{\epsilon}[v]=\frac{1}{2}\int_{\mathbb R^N} \left( \epsilon^2|\nabla v|^2+v^2\right)dx-\int_{\mathbb R^N} F(v)dx+\frac{1}{4}\int_{\mathbb R^N}\phi[v^2](x)v^2(x)dx
.\]

We can now state our results.
Our first theorem is about the existence of  nodal solutions whose form consists of one positive peak centered in $0$ surrounded by $k$ negative peaks located near the vertices of a regular polygon, with the number $k$ sufficiently large (see figure \ref{fig1}).

\

\begin{figure}[htbp]\label{fig1}
\centering

\includegraphics[scale=9]{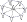}
\caption{A configuration with 1 positive peaks at the origin surrounded by 7 negative peaks.}
\end{figure}

\

\begin{theorem}\label{theorem poligono piano}
Let (f1)and (f2) hold  and let $N\in [3,6].$\\   Fix $k\geq 7$
and let $Q_1,\ldots,Q_k\in\mathbb R^2$ be the vertices of a two-dimensional convex regular polygon centered at $0.$ Then there exists $\epsilon_0>0$ such that for any $\epsilon\in (0,\epsilon_0)$, there is $r_{\epsilon}>0$ and a sign-changing solution $v_{\epsilon}\in H^1(\mathbb R^N)$ to \eqref{eqScalare} of the form

$$v_{\epsilon}(x)=w\left(\frac{x}{\epsilon}\right)-\sum_{i=1}^kw\left(\frac{x- {P_i}_{\epsilon}}{\epsilon}\right)+h.o.t.,\ \mbox{ as $\epsilon\rightarrow 0$}$$
uniformly for $x\in \mathbb R^N$.

Here ${P_i}_{\epsilon}:=(r_{\epsilon}Q_i,{\bf 0})\in\mathbb R^N,$ $i=1,\ldots,k$ and $r_{\epsilon}\rightarrow 0$ as $\epsilon\rightarrow 0.$ Moreover

$$\lim_{\epsilon\rightarrow 0}\frac{r_{\epsilon}}{\epsilon\log{\frac{1}{\epsilon^2}}}=C$$
for some $C>0.$
\end{theorem}

\

More in general we can prove the following result:
\begin{theorem}\label{theorem politopo}
Let $Q_1,\ldots,Q_k\in\mathbb R^h$ $(2\leq h\leq N)$ be the vertices of a convex regular polytope in $\mathbb R^h$ centered at $0$ and having radius $1$ and side $s.$

Assume $s\leq 1,$ $h>2$ or $s<1,$ $h=2.$
Then there exists $\epsilon_0>0$ such that for any $\epsilon\in (0,\epsilon_0)$, there is a sign-changing solution $v_{\epsilon}$ to \eqref{eqScalare} with one positive peak centered in $0$ and $k$ negative peaks centered at points ${P_i}_{\epsilon}:=(r_{\epsilon}Q_i,{\bf 0})\in\mathbb R^N,$ $i=1,\ldots,k,$ such that $r_{\epsilon}\rightarrow 0$ as $\epsilon\rightarrow  0.$ Moreover $$\lim_{\epsilon\rightarrow 0}\frac{r_{\epsilon}}{\epsilon\log{\frac{1}{\epsilon^2}}}=C$$
for some $C>0.$
\end{theorem}

The proof of our results is based on the well known Lyapunov-Schmidt reduction procedure (see \cite{AM}). In particular, in order to deal with nodal clustered solutions, we perform the reduction in suitable symmetric settings in the spirit of \cite{DP}. \\

We outline here the main ideas.\\
First our approximate solutions are constructed as the sum (with sign) of suitably rescaled $w$ centered at distinct points $P_i\in\mathbb R^N$ such that $P_i\rightarrow 0$ as $\epsilon\rightarrow 0.$

This choice is the most natural. In fact if $v$ is a solution of \eqref{eqScalare} and $P\in\mathbb R^N$ then
$v_{\epsilon}(x):=v(\epsilon x+P)$ solves the equation $-\Delta u+u+\epsilon^2\phi[u^2]u=f(u)$ which, since $\epsilon\rightarrow 0,$ can be approximated by problem \eqref{problema limite}. Hence it's quite natural to take $v\sim \pm w(\frac{x-P}{\epsilon})$  as a solution of \eqref{eqScalare} for $\epsilon$ small.

Moreover, if we take several different fixed points $P_i\in \mathbb R^N,$ then  $v\sim\sum_i \pm w(\frac{x-P_i}{\epsilon})$ is still a good approximation of a solution $v$ of \eqref{eqScalare} for $\epsilon$ small enough, in spite of the presence of nonlinear terms in the equation. The reason is that, thanks to the exponential decay of $w,$ the interactions among peaks centered at different fixed points becomes negligible when $\epsilon\rightarrow 0$.

In our case however we are looking for clustered solutions, namely the points $P_i\rightarrow 0$ as $\epsilon\rightarrow 0$. This means that the interactions among peaks play a role.

Anyway by locating the peaks in suitable symmetric configurations, still we will be able to find a solution of the desired form.

Indeed we recall that the Lyapunov-Schmidt reduction method reduces the problem to find a critical point for a functional defined on a finite-dimensional space (reduced functional). In our case the reduced functional, up to a positive constant, has the form
\begin{eqnarray*}
M_{\epsilon}[(P_1,..., P_{\ell})]&=& \epsilon^{2}\sum_{i\neq j}\frac{1}{|\frac{P_i-P_j}{\epsilon}|^{N-2}}- \sum_{i\neq j, \lambda_i=\lambda_j}w\left(\frac{P_i-P_j}{\epsilon}\right)\\
 &&+ \sum_{i\neq j, \lambda_i=-\lambda_j}w\left(\frac{P_i-P_j}{\epsilon}\right)
+\ \  h.o.t. ,
\end{eqnarray*}
where $\lambda_i=\pm 1$ according to the sign of each peak and the unknowns $P_i$ determine the location of the peaks. We notice that it consists of three main terms: the first term depends on the Poisson potential effect, the second term is due to the interplay between the peaks of the same sign and has a repulsive effect, the third term is due to the interaction between peaks of opposite sign and has an attractive effect.
\\
Observe that the first term $\epsilon^{2}\sum_{i\neq j}\frac{1}{|\frac{P_i-P_j}{\epsilon}|^{N-2}}$ increases when the points $P_i$ are close to zero, while using the exponential decay of $w$, the interaction term  $ - \sum_{i\neq j, \lambda_i=\lambda_j}w\left(\frac{P_i-P_j}{\epsilon}\right)$ increases when the mutual distance between the points $P_i$ is big.
 Hence, if we restrict the functional to suitable symmetric configurations in which the peaks having opposite sign are kept away from each other, then the mutual interaction between opposite peaks, i.e. the third term $+\sum_{i\neq j, \lambda_i=-\lambda_j}w\left(\frac{P_i-P_j}{\epsilon}\right),$ becomes negligible and so we can easily conclude that the equilibrium is achieved for a suitable (symmetric) configuration of the points $P_i,$ which is a local maximum for the functional $M_{\epsilon},$ namely we have produced a sign-changing cluster solution for the problem $(\mathcal{SP}).$\\\\
We remark that one can find solutions with analogous symmetric configurations as in all the previous results, also in the case of the scalar Schr\"{o}dinger equation $$-\epsilon^2\Delta v+ V(|x|)v=|v|^{p-1}v\qquad x\in \mathbb R^N$$ in presence of a radially symmetric potential $V$ with a local maximum in $0$ (see also \cite{DP}).

\

\begin{center}
{\bf Notations}
\end{center}
Before going on we establish some notations.\\
Let us denote by $H^1(\mathbb R^N)$ the usual Sobolev space endowed with scalar product and norm $$(u, v)_{\epsilon}=\int_{\mathbb R^N}\left(\epsilon^2\nabla u\nabla v+uv\right)\, dx; \qquad \|u\|^2:=\int_{\mathbb R^N}\left(\epsilon^2|\nabla u|^2+u^2\right)\, dx$$ and by $D^{1, 2}(\mathbb R^N)$ the completition of the space $C^{\infty}_c(\mathbb R^N)$ with respect to the norm $$\|u\|_{D^{1, 2}}:=\left(\int_{\mathbb R^N}\epsilon^2|\nabla u|^2\, dx \right)^{1/2}.$$ Moreover let $L^p(\mathbb R^N)$ the usual Lebesgue space endowed with the norm $$|u|_p:=\left(\int_{\mathbb R^N}|u|^p\, dx\right)^{\frac{1}{p}}\,\, p\in [1, \infty)\qquad \|u\|_{\infty}=\sup_{x\in \mathbb R^N}|u(x)|.$$ In particular, let us denote by $\langle \cdot , \cdot \rangle$ the usual scalar product in $L^2(\mathbb R^N)$, namely $$\langle u, v \rangle := \int_{\mathbb R^N}uv\, dx.$$

\

\section{General setting}
The Lyapunov-Schmidt reduction will be made around an appropriate set of approximating solutions. Precisely, for any $\ell\in \mathbb N$ we define 
\begin{equation}\label{gammeepsilon}\Gamma_{\epsilon}:=\left\{{\bf P}=(P_1, \ldots, P_{\ell})\in \mathbb R^{N{\ell}}\,\, :\,\, \beta^2\epsilon\log\frac{1}{\epsilon^2}<|P_i-P_j|<\epsilon\left(\log\frac{1}{\epsilon^2}\right)^2\,\, \mbox{for}\,\, i\neq j\right\}\end{equation} where $\beta\in (\sigma, 1)$ is sufficiently close to $1$. Let ${\bf P}\in \bar{\Gamma}_{\epsilon}$ and set $w_{P_{i}}(x)=w(\frac{x-P_{i}}{\epsilon})$, $i=1, \ldots, \ell$. We look for solutions of \eqref{eqScalare} of the form \begin{equation}\label{generalansatz}v_{\epsilon}(x):=w_{\bf P}(x)+\psi_{\epsilon}(x)\end{equation} where $\psi_{\epsilon}$ will be a remainder term belonging into a suitable space and the approximating solution $w_{\bf P}$ is of the form $$w_{\bf P}(x)=\sum_{i=1}^{\ell}\lambda_{i}w_{P_{i}}(x)$$ with $\lambda_{i}=\pm 1$ according to the sign of each peak.\\\\

In particular we will reduce ourselves to symmetric configurations, finding solutions $v_{\epsilon}$ with some symmetric properties. Here we show that $\phi[v_{\epsilon}^2]$ preserves the same symmetry property. Indeed, let $G$ be a group of symmetries of $\mathbb R^N$ and let $g\in G$. For $u:\mathbb R^N\rightarrow \mathbb R$ we set
 \begin{equation}\label{action}
 (T_gu)(x)=u(gx).
 \end{equation}
 Let $$X:=\left\{u\in H^1(\mathbb R^N)\,\, :\,\, T_g u=u,\,\, g\in G \right\}$$ and $$Y:=\left\{\phi\in D^{1, 2}(\mathbb R^N)\,\,:\,\, T_g\phi=\phi, \,\, g\in G \right\}.$$ We remark that $X, Y$ are the subspace of $H^1$ and $D^{1, 2}$ respectively invariant under the action (\ref{action}).
 \begin{lemma}\label{invact}
 If $u\in X$ then $\phi[u^2]\in Y$.
 \end{lemma}
 \begin{proof}
 Let $u\in X$. To prove that $\phi[u^2]\in Y$ we have to show that $\phi[u^2]$ is invariant under the action (\ref{action}). To this aim let us evaluate $$-\Delta(T_g \phi[u^2])=T_g(-\Delta \phi[u^2])=T_g(a_N u^2)=a_N u^2(gx)=a_N u^2(x)=-\Delta \phi[u^2]$$ and then, by the uniqueness of the solution, it follows that $T_g\phi[u^2]=\phi[u^2]$.
 \end{proof}
 Since we look for a solution near $w_{\bf P}$, a key step is to evaluate $\mathcal{S}(w_{{\bf P}})$
where $$\mathcal{S}(v):=\epsilon^2\Delta v -v+f(v)-\phi[v^2]v.$$ What we can prove is the following result (for the proof see for instance \cite{DW}):

\begin{lemma} \label{LemmaStimaErroreS}
Let $\beta\in (\sigma, 1).$ There exists a constant $C>0$ such that for every $\epsilon>0$ and ${\bf P}=(P_1,\ldots,P_l)\in \mathbb R^{N\ell}$ with $|P_i-P_j|\geq 2\beta^2\epsilon \log{\frac{1}{\epsilon}},$ for $i\neq j:$
$$|S_{\epsilon}[w_{{\bf P}}]|\leq  C \epsilon ^{\beta^2(\beta^2+\sigma)}\sum_{i=1}^lw_{P_i}^{1-\beta^2}.$$
\end{lemma}

\

\section{Energy estimates}\label{sectionReduced}
Let us fix $\ell\in\mathbb N$ and  ${\bf P}=(P_1,\ldots, P_{\ell})\in  \mathbb R^{N\ell}$. In this section we derive the following key result about the interaction among $\ell$ signed bumps displaced in $\bf P$.

\begin{proposition}\label{4.1} The following energy estimates hold as $\displaystyle{\left|\frac{P_i-P_j}{\epsilon}\right|}\rightarrow +\infty$:
\begin{equation}\label{stimaEnergiaSchrodinger}
\int_{\mathbb R^N} \left(\frac{1}{2}\left( \epsilon^2|\nabla w_{{\bf P}}|^2+w_{{\bf P}}^2\right)- F(w_{{\bf P}})\right)dx = \epsilon^N \ell I[w]-\epsilon^N (\gamma_0+o(1))\sum_{i\neq j}\lambda_i\lambda_jw\left(\frac{P_i-P_j}{\epsilon}\right),
\end{equation}
\begin{equation}\label{stimaEnergiaPoisson}
\frac{1}{4}\int_{\mathbb R^N}\!\!\int_{\mathbb R^N}\frac{w_{{\bf P}}^2(x)w_{{\bf P}}^2(y)}{|x-y|^{N-2}}dxdy=\epsilon^{N+2}C_1+\epsilon^{N+2}(C_2+o(1))\sum_{i\neq j}\frac{1}{|\frac{P_i-P_j}{\epsilon}|^{N-2}}
\end{equation}
where $\gamma_0:=\frac{1}{2}\int_{\mathbb R^N}f(w)e^{x_1}dx$
and $C_1, C_2$ are positive constants.
\end{proposition}
In order to prove Proposition \ref{4.1} we will need some useful lemmas that we briefly recall here.
From \cite[Lemma 3.1]{DW} one has
\begin{lemma}\label{Lemma 3.1}
For $i\neq j$
$$\int_{\mathbb R^N}f(w_{P_i})w_{P_j}=\epsilon^N w\left(\frac{P_i-P_j}{\epsilon}\right) (2\gamma_0+o(1))$$
where $\gamma_0$ is the same constant defined in Proposition \ref{4.1}.
\end{lemma}
Moreover in  \cite[pg. 23]{DP} it has been proved that
\begin{lemma}\label{Lemma H(P)}
Let
\[H({\bf P}):=   \int_{\mathbb R^N} [F(\sum_{i=1}^{\ell}\lambda_i w_{P_i})-\sum_{i=1}^{\ell}F(w_{P_i})-\sum_{i\neq j}\lambda_i\lambda_j f(w_{P_i})w_{P_j}]\, dx,\]
then
\[|H({\bf P})|=\epsilon^No(1)\sum_{i\neq j}w\left(\frac{P_i-P_j}{\epsilon}\right).\]
\end{lemma}
The following result can be found in  \cite{DW}:
\begin{lemma}\label{Lemma 3.2}
For every $\beta\in\{1,\ldots, N-1\}$ and $g:\mathbb R^N\rightarrow \mathbb R$ such that $(1+|y|^{\beta+1})g\in L^1\cap L^{\infty}$ set
\[\Psi_{\beta}[g](x):=\int_{\mathbb R^N}\frac{1}{|x-y|^{\beta}}g(y)dy.\]
Then there exist two positive constants $C(\beta, g), C'(\beta, g)$ such that
\[\left |\Psi_{\beta}[g] (x)-\frac{C(\beta,g)}{|x|^{\beta}}\right|\leq \frac{C'(\beta, g)}{|x|^{\beta +1}}.\]
\end{lemma}

Finally in order to estimate the energy term coming from the Poisson equation, we will also need the following:
\begin{lemma}\label{PoissonSolutionBounded}There exists a constant $C>0$ such that for every $P_i,P_j\in \mathbb R^N$ and every $x\in\mathbb R^N$
\begin{equation}\label{laPrima}\phi[w_{P_i}w_{P_j}](x)\leq \epsilon^2 C.\end{equation}
\begin{equation}\label{laSeconda}\phi[w_{P_i}](x)\leq \epsilon^2 C.\end{equation}
\end{lemma}

\begin{proof}
We prove \eqref{laPrima}, the proof of \eqref{laSeconda} is similar. If $i=j$ one has
\[\phi[w_{P_i}^2](x)=\int_{\mathbb R^N} \frac{w_{P_i}^2(y)}{|x-y|^{N-2}}dy=\epsilon^2\int_{\mathbb R^N} \frac{w^2(y)}{|\frac{x-P_i}{\epsilon}-y|^{N-2}}dy=\epsilon^2\phi[w^2]\left(\frac{x-P_i}{\epsilon}\right)\]
where $x\mapsto \phi[w^2](x)$ is bounded by Strauss Lemma since it belongs to $C^2\cap D^{1,2}_r$.\\
When $i\neq j$ then
\begin{eqnarray*}
\phi[w_{P_i}w_{P_j}](x)&=&\int_{\mathbb R^N} \frac{w_{P_i}(y)w_{P_j}(y)}{|x-y|^{N-2}}dy\\
&&\hskip-2.0cm \leq w\left(\frac{1}{2}\frac{|P_i-P_j|}{\epsilon}\right)\left(\int_{\{|y-P_j|\geq\frac{1}{2}|P_i-P_j|\}}\frac{w_{P_i}(y)}{|x-y|^{N-2}}dy+\int_{\{|y-P_i|\geq\frac{1}{2}|P_i-P_j|\}}\frac{w_{P_j}(y)}{|x-y|^{N-2}}dy\right)\\
&&\hskip-2.0cm \leq \epsilon^2 w\left(\frac{1}{2}\frac{|P_i-P_j|}{\epsilon}\right)
\left(\int\frac{w(y)}{|\frac{x-P_i}{\epsilon}-y|^{N-2}}dy+\int\frac{w(y)}{|\frac{x-P_j}{\epsilon}-y|^{N-2}}dy\right)\\
&&\hskip-2.0cm = \epsilon^2 w\left(\frac{1}{2}\frac{|P_i-P_j|}{\epsilon}\right)\left(\phi[w]\left(\frac{x-P_i}{\epsilon}\right)+\phi[w]\left(\frac{x-P_j}{\epsilon}\right)\right),
\end{eqnarray*}
and $x\mapsto w(x), \phi[w](x)$ are bounded ($\phi[w]\in C^2\cap D^{1,2}_r$)).
\end{proof}

\begin{proof}[Proof of Proposition \ref{4.1}]
We first prove \eqref{stimaEnergiaSchrodinger}.

Easy computations show that
\begin{eqnarray*}
&& \hskip-0.7cm \int_{\mathbb R^N} |\nabla w_{{\bf P}}|^2\,dx+\int_{\mathbb R^N} w_{{\bf P}}^2\, dx = \int_{\mathbb R^N} |\sum_{i=1}^{\ell}\lambda_i\nabla w_{P_i}|^2\,dx+\int_{\mathbb R^N} (\sum_{i=1}^{\ell}\lambda_i w_{P_i})^2\, dx\\
&&=  \epsilon^{N-2}{\ell}\int_{\mathbb R^N} |\nabla w|^2\, dx+ \sum_{i\neq j}\lambda_i\lambda_j\int \nabla w_{P_i}\nabla w_{P_j}\,dx,+\epsilon^N{\ell}\int_{\mathbb R^N}  w^2\, dx+ \sum_{i\neq j}\lambda_i\lambda_j\int_{\mathbb R^N} w_{P_i} w_{P_j}\, dx,
\end{eqnarray*}
\begin{equation*}
\int_{\mathbb R^N} F(w_{{\bf P}})\, dx =  \int_{\mathbb R^N} F(\sum_{i=1}^{\ell}\lambda_i w_{P_i})\, dx
= \epsilon^N {\ell}\int_{\mathbb R^N} F(w)\, dx  +\sum_{i\neq j}\lambda_i\lambda_j \int_{\mathbb R^N} f(w_{P_i})w_{P_j}\, dx +H({\bf P}).
\end{equation*}
Hence, combining the previous estimates and using Lemma \ref{Lemma H(P)} we obtain \eqref{stimaEnergiaSchrodinger},
and the conclusion follows applying  Lemma \ref{Lemma 3.1}.
\\\\
Next we prove \eqref{stimaEnergiaPoisson}. An easy computation shows that
\begin{eqnarray*}
\int_{\mathbb R^N}\!\!\int_{\mathbb R^N}\frac{w_{{\bf P}}^2(x)w_{{\bf P}}^2(y)}{|x-y|^{N-2}}\,dx\,dy &=& \sum_{i=1}^{\ell}\int_{\mathbb R^N}\!\!\int_{\mathbb R^N}\frac{ w_{P_i}(x)^2w_{P_i}(y)^2}{|x-y|^{N-2}}\,dx\,dy\\
&& + \sum_{i\neq j}^{\ell}\int_{\mathbb R^N}\!\!\int_{\mathbb R^N}\frac{ w_{P_i}(x)^2w_{P_j}(y)^2}{|x-y|^{N-2}}\,dx\,dy\\
&& + \,2\sum_{i\neq j, h}^{\ell}\lambda_i\lambda_j\int_{\mathbb R^N}\!\!\int_{\mathbb R^N}\frac{ w_{P_i}(x)w_{P_j}(x)w_{P_h}(y)^2}{|x-y|^{N-2}}\,dx\,dy\\
&& + \sum_{i\neq j, h\neq k}^{\ell}\lambda_i\lambda_j\lambda_h\lambda_k\int_{\mathbb R^N}\!\!\int_{\mathbb R^N}\frac{ w_{P_i}(x)w_{P_j}(x)w_{P_h}(y)w_{P_k}(y)}{|x-y|^{N-2}}\,dx\,dy.
\end{eqnarray*}
We evaluate each term in the RHS. Indeed
\begin{eqnarray*}\int_{\mathbb R^N}\!\!\int_{\mathbb R^N}\frac{ w_{P_i}(x)^2w_{P_i}(y)^2}{|x-y|^{N-2}}\,dx\,dy=
\epsilon^{N+2}\int_{\mathbb R^N}\!\!\int_{\mathbb R^N}\frac{ w(x)^2w(y)^2}{|x-y|^{N-2}}\,dx\,dy.
\end{eqnarray*}
For $i\neq j,$ by using Lemma \ref{Lemma 3.2} twice we estimate
\begin{eqnarray*}
&& \hskip-1.3cm \int_{\mathbb R^N}\!\!\int_{\mathbb R^N}\frac{ w_{P_i}(x)^2w_{P_j}(y)^2}{|x-y|^{N-2}}\,dx\,dy  
=\epsilon^2\int_{\mathbb R^N} w_{P_i}(x)^2\Psi_{N-2}[w^2]\left(\frac{x-P_j}{\epsilon}\right)\,dx\\
&& = \epsilon^2C \int_{\mathbb R^N} w_{P_i}(x)^2\frac{1}{\left|\frac{x-P_j}{\epsilon}\right|^{N-2}}\,dx+\epsilon^2O(1)\int_{\mathbb R^N} w_{P_i}(x)^2\frac{1}{\left|\frac{x-P_j}{\epsilon}\right|^{N-1}}\,dx\\
%
&&=\epsilon^{N+2}C\Psi_{N-2}[w^2]\left( \frac{P_j-P_i}{\epsilon}\right)+\epsilon^{N+2}O(1)\Psi_{N-1}[w^2]\left( \frac{P_j-P_i}{\epsilon}\right)\\
%
&&= \epsilon^{N+2} C^2(1+o(1))\frac{1}{|\frac{P_j-P_i}{\epsilon}|^{N-2}}
\end{eqnarray*}
where $C=C(N-2, w^2)$ is the positive constant in Lemma \ref{Lemma 3.2}.\\
Finally for $i\neq j$ and $h,k\in\{1,\ldots ,{\ell}\},$ by using \eqref{laPrima} in Lemma \ref{PoissonSolutionBounded} and the exponential decay of $w$ we have
\begin{eqnarray*}&&\int_{\mathbb R^N}\!\!\int_{\mathbb R^N}\frac{ w_{P_i}(x)w_{P_j}(x)w_{P_h}(y)w_{P_k}(y)}{|x-y|^{N-2}}\,dx\,dy
=\int_{\mathbb R^N} w_{P_i}(x)w_{P_j}(x) \phi[w_{P_h}w_{P_k}](x)\,dx\\
&&\leq \epsilon^2 C \int _{\mathbb R^N}w_{P_i}(x)w_{P_j}(x)\,dx\\
&&\\
&&\leq \epsilon^2 Cw\left(\frac{1}{2}\frac{|P_i-P_j|}{\epsilon}\right)\left(\int_{\{|x-P_j|\geq\frac{1}{2}|P_i-P_j|\}} w_{P_i}(x)\, dx+\int_{\{|x-P_i|\geq\frac{1}{2}|P_i-P_j|\}} w_{P_j}(x)\, dx\right)
\\
&& \leq \epsilon^{N+2} Cw\left(\frac{1}{2}\frac{|P_i-P_j|}{\epsilon}\right)
\int_{\mathbb R^N} w(x)\,dx
\leq \epsilon^{N+2} C w\left(\frac{1}{2}\frac{|P_i-P_j|}{\epsilon}\right)
 = \epsilon^{N+2} o(1)\frac{1}{|\frac{P_i-P_j}{\epsilon}|^{N-2}}
\end{eqnarray*}
Estimate \eqref{stimaEnergiaPoisson} is obtained as a  combination of all the previous estimates.
\end{proof}

\

\section{The linearized problem}\label{section linearized}

First we need the following result based on PDE estimates.
\begin{lemma}\label{PDEest}
Let $\epsilon>0$, ${\bf P}\in \Gamma_{\epsilon}$ and $v\in C^2(\mathbb R^N)$.  There exists $\mu_0$ (independent on $\epsilon, {\bf P}$ and $v$) such that if $\mu \in (0, \mu_0)$ and
\begin{equation}\label{hip}
\begin{array}{lr}
\left|\epsilon^2\Delta v - (1+ \phi[{w_{\bf P}}^2])v\right|\leq c_0 e^{-\mu \min_{i=1, \ldots, \ell}\frac{|x-P_{i}|}{\epsilon}},\qquad \forall\,\, x\in \mathbb R^N, \\ \\
\,\,\, v(x)\rightarrow 0\,\ \mbox{as}\,\,|x|\rightarrow +\infty
\end{array}
\end{equation}
for some $c_0>0$, then
$$|v(x)|\leq 2c_0(\ell-1+e^2)e^{-\mu\min_{i=1, \ldots, \ell}\frac{|x-P_{i}|}{\epsilon}}, \qquad \forall\,\, x\in \mathbb R^N.$$
\end{lemma}
\begin{proof}
We use a comparison principle. Take a $\chi(t)$ a smooth cut-off function such that $$\chi(t)=1, \,\, \mbox{for}\,\,\ |t|\leq 1, \qquad \chi(t)=0\,\, \mbox{for}\,\, |t|\geq 2, \,\,\, 0\leq \chi \leq 1.$$ Now consider the following auxiliary function:
$$\xi(x)=2c_0 \sum_{i=1}^\ell \left[e^{-\mu\frac{|x-P_{i}|}{\epsilon}}+(1-e^{-\mu\frac{|x-P_{i}|}{\epsilon}})\chi\left(\mu\frac{|x-P_{i}|}{\epsilon}\right)\right].$$
We have that
\begin{eqnarray*}
\Delta \xi&=& 2c_0 \sum_{i}\left[\frac{\mu^2}{\epsilon^2} e^{-\mu\frac{|x-P_{i}|}{\epsilon}}+\frac{(N-1)\mu}{\epsilon^2}\frac{\epsilon}{|x-P_{i}|}e^{-\mu\frac{|x-P_{i}|}{\epsilon}}+\frac{\mu^2}{\epsilon^2}(1-e^{-\mu\frac{|x-P_{i}|}{\epsilon}})\chi''\left(\mu\frac{|x-P_{i}|}{\epsilon}\right)\right.\\
&&+\frac{(N-1)\mu}{\epsilon^2}\frac{\epsilon}{|x-P_{i}|}(1-e^{-\mu\frac{|x-P_{i}|}{\epsilon}})\chi'\left(\mu\frac{|x-P_{i}|}{\epsilon}\right)+\frac{2\mu^2}{\epsilon^2}e^{-\mu\frac{|x-P_{i}|}{\epsilon}}\chi'(\mu\frac{|x-P_{i}|}{\epsilon})\\
&&\left. +\frac{(N-1)\mu}{\epsilon^2}\frac{\epsilon}{|x-P_{i}|}e^{-\mu\frac{|x-P_{i}|}{\epsilon}}\chi\left(\mu\frac{|x-P_{i}|}{\epsilon}\right)-\frac{\mu^2}{\epsilon^2}e^{-\mu\frac{|x-P_{i}|}{\epsilon}}\chi\left(\mu\frac{|x-P_{i}|}{\epsilon}\right)\right]
\end{eqnarray*}
Hence
\begin{eqnarray*}
\epsilon^2\Delta \xi&=& 2c_0 \sum_{i}\left[\mu^2 e^{-\mu\frac{|x-P_{i}|}{\epsilon}}+(N-1)\mu\frac{\epsilon}{|x-P_{i}|}e^{-\mu\frac{|x-P_{i}|}{\epsilon}}+\mu^2(1-e^{-\mu\frac{|x-P_{i}|}{\epsilon}})\chi''\left(\mu\frac{|x-P_{i}|}{\epsilon}\right)\right.\\
&&+(N-1)\mu\frac{\epsilon}{|x-P_{i}|}(1-e^{-\mu\frac{|x-P_{i}|}{\epsilon}})\chi'\left(\mu\frac{|x-P_{i}|}{\epsilon}\right)+2\mu^2e^{-\mu\frac{|x-P_{i}|}{\epsilon}}\chi'(\mu\frac{|x-P_{i}|}{\epsilon})\\
&&\left. +(N-1)\mu\frac{\epsilon}{|x-P_{i}|}e^{-\mu\frac{|x-P_{i}|}{\epsilon}}\chi\left(\mu\frac{|x-P_{i}|}{\epsilon}\right)-\mu^2e^{-\mu\frac{|x-P_{i}|}{\epsilon}}\chi\left(\mu\frac{|x-P_{i}|}{\epsilon}\right)\right]
\end{eqnarray*}
Fixed $x\in \mathbb R^N$. We distinguish three cases:
\begin{itemize}
\item[1.] There exists $i\in \{1, \ldots, \ell\}$ such that $\mu\frac{|x-P_{i}|}{\epsilon}\leq 1$. Then
\begin{eqnarray*}
\xi(x)&=&2c_0+2c_0\sum_{j\neq i}e^{-\mu\frac{|x-P_{j}|}{\epsilon}}\\
&=& 2c_0 e^{-\mu\frac{|x-P_{i}|}{\epsilon}}e^{\mu\frac{|x-P_{i}|}{\epsilon}} +2c_0\sum_{j\neq i}e^{-\mu\frac{|x-P_{j}|}{\epsilon}}\\
&\leq & 2c_0 (e+\ell-1)e^{-\mu\frac{|x-P_{i}|}{\epsilon}}
\end{eqnarray*}
and, since,
\begin{equation}\label{deltae}
\epsilon^2\Delta e^{-\mu|x-P_{i}|}=\left(\mu^2-\frac{(N-1)\mu\epsilon}{|x-P_{i}|}\right)e^{-\mu\frac{|x-P_{i}|}{\epsilon}}
\end{equation}
 for $x\neq P_i$ we have
\begin{eqnarray*}
\epsilon^2\Delta\xi-(1+\phi[{w_{\bf P}}^2]\xi&\leq & 2c_0\sum_{j\neq i}(\mu^2-1)e^{-\mu\frac{|x-P_{j}|}{\epsilon}}-2c_0\\
&\leq & -c_0e^{-\mu\frac{|x-P_{i}|}{\epsilon}}
\end{eqnarray*}
provided $\mu$ is sufficiently small.
\item[2. ] For all $i=1, \ldots, \ell$, $\mu\frac{|x-P_{i}|}{\epsilon}\geq 2$. Then $$\xi(x)=2c_0\sum_{i}e^{-\mu\frac{|x-P_{i}|}{\epsilon}}\leq 2c_0 \ell e^{-\mu\min_{i=1, \ldots, \ell}\frac{|x-P_{i}|}{\epsilon}}$$ and by (\ref{deltae})
\begin{eqnarray*}
\epsilon^2\Delta\xi-(1+\phi[{w_{\bf P}}^2])\xi&\leq & 2c_0\sum_{i}(\mu^2-1)e^{-\mu\frac{|x-P_{i}|}{\epsilon}}\\
&\leq & -c_0\sum_{\theta}e^{-\mu\frac{|x-P_{i}|}{\epsilon}}\leq -c_0 e^{-\mu\min_{i=1, \ldots \ell}\frac{|x-P_{i}|}{\epsilon}}
\end{eqnarray*}
provided $\mu$ is sufficiently small.
\item[3. ] There exists $i\in \{1, \ldots, \ell\}$ such that $1<\mu\frac{|x-P_{i}|}{\epsilon}<2$. Then
\begin{eqnarray*}
\xi(x)&=& 2c_0\left[e^{-\mu\frac{|x-P_{i}|}{\epsilon}}+(1-e^{-\mu\frac{|x-P_{i}|}{\epsilon}})\chi\left(\mu\frac{|x-P_{i}|}{\epsilon}\right)+\sum_{j\neq i}e^{-\mu\frac{|x-P_{j}|}{\epsilon}}\right]\\
&\leq & 2c_0\left[1+(\ell-1)e^{-\mu\frac{|x-P_{i}|}{\epsilon}}\right]\\
&=& 2c_0(e^{\mu\frac{|x-P_{i}|}{\epsilon}}+\ell-1)e^{-\mu\frac{|x-P_{i}|}{\epsilon}}\\
&\leq & 2c_0(e^2+\ell-1)e^{-\mu\frac{|x-P_{i}|}{\epsilon}}
\end{eqnarray*}
and, since
\begin{eqnarray*}
&& \hskip-1.0cm \epsilon^2\Delta\left[(1-e^{-\mu\frac{|x-P_{i}|}{\epsilon}})\chi\left(\mu\frac{|x-P_{i}|}{\epsilon}\right)\right]=
\mu^2(1-e^{-\mu\frac{|x-P_{i}|}{\epsilon}})\chi''\left(\mu\frac{|x-P_{i}|}{\epsilon}\right)\\
&&+\frac{(N-1)\mu\epsilon}{|x-P_{i}|}(1-e^{-\mu\frac{|x-P_{i}|}{\epsilon}})\chi'\left(\mu\frac{|x-P_{i}|}{\epsilon}\right)+2\mu^2e^{-\mu\frac{|x-P_{i}|}{\epsilon}}\chi'\left(\mu\frac{|x-P_{i}|}{\epsilon}\right)\\
&& + \frac{(N-1)\mu\epsilon}{|x-P_i|}e^{-\mu\frac{|x-P_{i}|}{\epsilon}}\chi\left(\mu\frac{|x-P_{i}|}{\epsilon}\right)-\mu^2e^{-\mu\frac{|x-P_{i}|}{\epsilon}}\chi\left(\mu\frac{|x-P_{i}|}{\epsilon}\right)\\
&&= O(\mu^2)e^{-\mu\frac{|x-P_{i}|}{\epsilon}}
\end{eqnarray*}
we have, by (\ref{deltae}),
\begin{eqnarray*}
\epsilon^2\Delta\xi-(1+\phi[w_{{\bf P}}^2])\xi &\leq & O(\mu^2)e^{-\mu\frac{|x-P_{i}|}{\epsilon}}-\xi\\
&\leq & (O(\mu^2)-2c_0)e^{-\mu\frac{|x-P_{i}|}{\epsilon}}\\
&\leq & -c_0 e^{-\mu\frac{|x-P_{i}|}{\epsilon}}
\end{eqnarray*}
for $\mu$ sufficiently small.
\end{itemize}
Hence, in any case, we have
\begin{equation}\label{xi}
\xi(x)\leq 2c_0 (e^2+\ell-1)e^{-\mu\min_{i=1, \ldots, \ell}\frac{|x-P_{i}|}{\epsilon}}
\end{equation}
\begin{equation}\label{deltaxi}
\epsilon^2\Delta \xi -(1+\phi_{w_{\bf P}})\xi \leq -c_0 e^{-\mu\min_{i=1, \ldots, \ell}\frac{|x-P_{i}|}{\epsilon}}
\end{equation}
for any $x\in \mathbb R^N$.\\
By (\ref{hip})
\begin{equation}\label{star}
-\epsilon^2\Delta v + (1+\phi[{w_{\bf P}}^2])v \leq |\epsilon^2\Delta v - (1+\phi[{w_{\bf P}}^2])v|\leq c_0 e^{-\mu\min_{i=1, \ldots, \ell}\frac{|x-P_{i}|}{\epsilon}}
\end{equation}
for all $x\in \mathbb R^N$. Then, by (\ref{deltaxi}) and (\ref{star}) we get
$$\epsilon^2\Delta (\xi-v)-(1+\phi[w_{{\bf P}}^2])(\xi-v)\leq 0$$ for all $x\in \mathbb R^N$.\\ We claim that $\xi-v \geq 0$ in $\mathbb R^N$.\\ Indeed, if we suppose by contradiction that the minimum point $\bar{x}$ of $\xi-v$ is such that $(\xi-v)(\bar{x})<0$, since $\Delta (\xi-v)(\bar{x})\geq 0$ then $$\epsilon^2\Delta (\xi-v)(\bar{x})-(1+\phi[w_{\bf P}^2])(\xi-v)(\bar{x})>0.$$ Analogously we can prove that $v+\xi\geq 0$. Thus $|v|\leq \xi$ and, using (\ref{xi}) we can conclude.
\end{proof}Let ${\bf P}\in \Gamma_{\epsilon}$. Let us introduce the following functions $$Z_{P_{i}, j}=f'(w_{P_{i}})\frac{\partial w_{P_{i}}}{\partial x_j}, \qquad i\in \{1, \ldots, \ell\},\quad j\in \{1, \ldots, N\}.$$ Since $$Z_{P_{i}, j}=-\epsilon^2\Delta \frac{\partial w_{P_{i}}}{\partial x_j}+\frac{\partial w_{P_{i}}}{\partial x_j}$$ after an integration by parts it is immediate to prove that
\begin{equation}\label{H1L2}
\left(v, \frac{\partial w_{P_{i}}}{\partial x_j}\right)_{\epsilon}=\langle v, Z_{P_{i}, j}\rangle\qquad \forall\,\, v \in H^1(\mathbb R^N).
 \end{equation}
 Then orthogonality to the functions $\frac{\partial w_{P_{i}}}{\partial x_j}$ in $H^1(\mathbb R^N)$ is equivalent to orthogonality to $Z_{P_{i}, j}$ in $L^2(\mathbb R^N)$. Then we easily get
\begin{equation}\label{prodotto}
\langle Z_{P_{i}, j}, \frac{\partial w_{P_{m}}}{\partial x_n}\rangle =\left(\frac{\partial w_{P_{i}}}{\partial x_j}, \frac{\partial w_{P_{m}}}{\partial x_n}\right)_{\epsilon}=\left\{
\begin{array}{lr}
\epsilon^{N-2}\left\|\frac{\partial w}{\partial x_1}\right\|^2_{H^1(\mathbb R^N)}\qquad \mbox{for}\,\, (i, j)=(m, n)\\
o(\epsilon^{N-2}) \qquad \qquad \mbox{for}\,\, (i, j)\neq (m, n)
\end{array}
\right.
\end{equation}
as $\epsilon \rightarrow 0$.\\ Let $\mu>0$ a sufficiently small number. We introduce the following weighted norm: $$\|v\|_{*, {\bf P}}:=\sup_{x\in \mathbb R^N} e^{\mu \min_{i=1, \ldots, \ell}\frac{|x-P_{i}|}{\epsilon}}|v(x)|,$$ and the spaces $$C_{*, {\bf P}}=\{v\in C(\mathbb R^N)\,\, :\,\,\, \|v\|_{*, {\bf P}}<\infty \}, \qquad H^2_{*, {\bf P}}=H^2(\mathbb R^N)\cap C_{*, {\bf P}}.$$ We consider the following linear problem : \\
Taken ${\bf P}\in \Gamma_{\epsilon}$  and given $h\in C_{*, {\bf P}}$ find a function $v$ and constants $c_{i, j}$ satisfying
\begin{equation}\label{problin}
\left\{
\begin{array}{lr}
\mathcal{L}_{{\bf P}}[v]=h+\sum_{i, j} c_{i, j}Z_{P_{i}, j}\\\\
v\in H^2_{*, {\bf P}},\qquad \langle v, Z_{P_{i}, j}\rangle =0, \,\, i=1, \ldots, \ell,\,\, j=1,\ldots,N.
\end{array}
\right.
\end{equation}
where $$\mathcal{L}_{{\bf P}}[v]:=\epsilon^2\Delta v -v+ f'(w_{{\bf P}})v-\phi[{w_{{\bf P}}^2}]v-2\phi[w_{{\bf P}}v]w_{{\bf P}}.$$
\begin{lemma}\label{Estnorm}
There exists $C>0$ such that, provided $\epsilon$ is sufficiently small, if ${\bf P}\in \bar{\Gamma}_{\epsilon}$ and $(v, c_{i, j}, h)$ satisfies (\ref{problin}) the following holds $$\|v\|_{*, {\bf P}}\leq C \|h\|_{*, {\bf P}}.$$
\end{lemma}
\begin{proof}
By contradiction, we assume the existence of a sequence $\epsilon_n\rightarrow 0$, $$(\bar{v}_n, \bar{c}_{i, j}^n)\in H^2_{*, {\bf P}^n}\times \mathbb R,\qquad \bar{h}_n\in C_{*, {\bf P}^n}$$ satisfying (\ref{problin}) such that $$\|\bar{v}_n\|_{*, {\bf P}^n}> n \|\bar{h}_n\|_{*, {\bf P}^n}.$$ Set $$v_n=\frac{\bar{v}_n}{\|\bar{v}_n\|_{*, {\bf P}^n}},\qquad c_{i, j}^n=\frac{\bar{c}^n_{i, j}}{\|\bar{v}_n\|_{*, {\bf P}^n}},\qquad h_n=\frac{\bar{h}_n}{\|\bar{v}_n\|_{*, {\bf P}^n}}.$$ We obtain that $(v_n, c_{i, j}^n, h_n)$ satisfies (\ref{problin}) and $$\|v_n\|_{*, {\bf P}^n}=1, \qquad \|h_n\|_{*, {\bf P}^n}=o(1).$$ Choose $(h', m)\in \{1, \ldots, \ell\}\times \{1, \ldots, N\}$ be such that, up to a subsequence, $|c^n_{h', m}|\geq |c^n_{i, j}|$ for all $(i, j)$ and $n$. By multiplying the equation in (\ref{problin}) by $\frac{\partial w_{P_{h'}^n}}{\partial x_m}$ and integrating on $\mathbb R^N$, we get
$$\underbrace{\sum_{i, j}c^n_{i, j}\langle Z_{P^n_{i}, j},\frac{\partial w_{P_{h'}^n}}{\partial x_m}\rangle }_{(A)}=-\underbrace{\langle h_n, \frac{\partial w_{P_{h'}^n}}{\partial x_m}\rangle}_{(B)}+\underbrace{\langle \mathcal{L}_{{\bf P}^n}[v_n], \frac{\partial w_{P_{h'}^n}}{\partial x_m}\rangle}_{(C)}.$$ First let us examine the term (A). By (\ref{prodotto}) $$(A)= \epsilon^{N-2}_n c^n_{h', m}\left(\left\|\frac{\partial w}{\partial x_1}\right\|^2+o(1)\right).$$ The term (B) can be estimated as
$$|(B)|=\left|\int_{\mathbb R^3}h_n \frac{\partial w_{P_{h'}^n}}{\partial x_m}\, dx\right|\leq \|h_n\|_{*, {\bf P}^n}\int_{\mathbb R^3}|\nabla w_{P^n_{h'}}|\, dx\leq \epsilon_n^{N-1}\|h_n\|_{*, {\bf P}^n}.$$
Then, as regards the last term (C) we find
\begin{eqnarray*}
|(C)|&=&\left|\int_{\mathbb R^N}\mathcal{L}_{{\bf P}^n}[v_n]\frac{\partial w_{P_{h'}^n}}{\partial x_m}\, dx\right|\\
&=&\left|\int_{\mathbb R^N}\left[\epsilon^2_n\Delta v_n -v_n+ f'(w_{{\bf P}^n})v_n-\phi[{w_{{\bf P}^n}^2}]v_n-2\phi[w_{{\bf P}^n} v_n] w_{{\bf P}^n}\right]\frac{\partial w_{P_{h'}^n}}{\partial x_m}\, dx\right|\\
&\leq & C\|v_n\|_{*, {\bf P}^n}\int_{\mathbb R^N}\left|f'(w_{{\bf P}^n})-f'(w_{P_{h'}^n})\right|\left|\frac{\partial w_{P_{h'}^n}}{\partial x_m}\right|\, dx + C\epsilon_n^{N-1}\|v_n\|_{*, {\bf P}^n}\\
& \leq & C\epsilon_n^{-1}\|v_n\|_{*, {\bf P}^n}\sum_{i\neq j}\int_{\mathbb R^N}w_{P_i^n}^{\sigma}w_{P_j^n}\, dx+ C\epsilon_n^{N-1}\|v_n\|_{*, {\bf P}^n}\\
&\leq & \ C\|v_n\|_{*, {\bf P}^n}\left(\epsilon_n^{N-1+2\beta^2\sigma}+\epsilon^{N-1}\right).
\end{eqnarray*}
Putting together (A), (B) and (C) we get $$c^n_{i, j}=o(\epsilon)\qquad \forall\,\,\ (i, j)$$ by which $$\|h_n + \sum_{i, j}c^n_{i, j}Z_{P_{i}^n, j}\|_{*, {\bf P}^n}=o(1).$$ This implies
\begin{equation}\label{pallino}
\|\mathcal{L}_{{\bf P}^n}[v_n]\|_{*, {\bf P}^n}=o(1).
\end{equation}

Fix $R>0$. We claim $$\|v_n\|_{L^{\infty}\left(\cup_{i=1}^\ell B_{R\epsilon_n} (P_{i}^n)\right)}=o(1).$$ Otherwise, we may assume that $$\|v_n\|_{L^{\infty}(B_{R\epsilon_n}(P_1^n))}\geq c>0$$ for some $R>0$. By multiplying the equation in (\ref{problin}) by $v_n$ and integrating by parts we get that the sequence $v_n(\epsilon_n x +P_1^n)$ is bounded in $H^1(\mathbb R^N)$. Therefore, possibly passing to a subsequence, $v_n(\epsilon_n x+P_1^n) \rightarrow v_0$ weakly in $H^1(\mathbb R^N)$ and a. e. in $\mathbb R^N$ and $v_0$ satisfies $$\Delta v_0-v_0+f'(w)v_0=0, \qquad |v_0(x)|\leq e^{-\mu|x|}.$$ According to elliptic regularity theory we may assume $v_n(\epsilon_n x +P_1^n) \rightarrow v_0$ uniformly on compact sets. Then $\|v_0\|_{\infty}\geq c$. By the non-degeneracy of $w$ (assumption (f2)) it follows $$v_0=\sum_{j=1}^N a_j \frac{\partial w}{\partial x_j}.$$ On the other hand, for $j\in \{1, \ldots, N\}$
\begin{eqnarray*}
&& \hskip-1.0cm 0=\int_{\mathbb R^N} v_n(\epsilon_n x +P_1^n) Z_{P_1^n, m}(\epsilon_n x+P_1^n)\, dx = \left(v_n(\epsilon x+P_1^n), \frac{\partial w_{P_1^n}}{\partial x_m}(\epsilon_n x+ P_1^n)\right)_{\epsilon}\\
&& \rightarrow \left(v_0, \frac{\partial w}{\partial x_m}\right)_1=a_m \left\|\frac{\partial w}{\partial x_1}\right\|^2.
\end{eqnarray*}
from which it follows $a_m=0$ and hence, in particular, $v_0=0$, a contradiction. Then the claim follows. We immediately obtain $$\|f'(w_{{\bf P}^n})v_n\|_{*, {\bf P}^n}=o(1)$$ and by (\ref{pallino}) $$\|\epsilon_n^2\Delta v_n -(1+\phi[{w_{{\bf P}^n}^2}])v_n\|_{*, {\bf P}^n}=o(1)$$ By the Lemma \ref{PDEest} we get $$v_n(x)=o(1)e^{-\mu \min_{i=1, \ldots, \ell}\frac{|x-P_{i}^n|}{\epsilon_n}}$$ which is a contradiction since $\|v_n\|_{*, {\bf P}^n}=1$.
\end{proof}
\begin{lemma}\label{EsUn}
For $\epsilon>0$ sufficiently small, for ${\bf P}\in \bar{\Gamma}_{\epsilon}$ and $h\in C_{*, {\bf P}}$ there exists a unique pair $(v, c_{i, j})\in H^2_{*, {\bf P}}\times \mathbb R^{N\ell}$ solving (\ref{problin}). Furthermore by Lemma (\ref{Estnorm}) $$\|v\|_{*, {\bf P}}\leq C \|h\|_{*, {\bf P}}.$$
\end{lemma}
\begin{proof}
The existence follows from the Fredholm alternative. To this aim, for every ${\bf P}\in \bar{\Gamma}_{\epsilon}$, let us consider $$W:=\left\{v\in H^1(\mathbb R^N)\,\, :\,\,\ \left(v, \frac{\partial w_{P_{i}}}{\partial x_j}\right)_{\epsilon}=0, \,\, i=1, \ldots, \ell\,\, j=1, \ldots, N \right\}.$$ It is easy to see that $W$ is a closed subset of $H^1(\mathbb R^N)$. By (\ref{H1L2}) $v\in W$ solves the equation in (\ref{problin}) if and only if $$\langle \mathcal{L}_{{\bf P}}[v], z\rangle =\langle h, z \rangle \qquad \forall\,\, z\in W.$$ Indeed, once we know $v$, we can determine the unique $c_{i, j}$ from the linear system of equations
\begin{equation}\label{sy}
\langle \mathcal{L}_{{\bf P}}[v], \frac{\partial w_{P_{i}}}{\partial x_j}\rangle = \langle h,  \frac{\partial w_{P_{i}}}{\partial x_j}\rangle +\sum_{m, n}c_{m, n}\langle Z_{P_{m}, n},  \frac{\partial w_{P_{i}}}{\partial x_j}\rangle.
\end{equation}
with $j=1, \ldots, N$ and $i=1, \ldots, \ell$. \\ The system (\ref{sy}) is equivalent to
\begin{eqnarray}\label{syeq}
\nonumber
&&\hskip-1.0cm -\int_{\mathbb R^N}f'(w_{{\bf P}}) \frac{\partial w_{P_{i}}}{\partial x_j} v\, dx +\int_{\mathbb R^N}\left(\phi[{w_{{\bf P}}^2}]+2\phi[w_{{\bf P}}v] \right) \frac{\partial w_{P_{i}}}{\partial x_j}v\, dx =\\
&&=-\int_{\mathbb R^N}h \frac{\partial w_{P_{i}}}{\partial x_j}\, dx -\sum_{m, n}c_{m, n}\int_{\mathbb R^N}Z_{P_{m}, n} \frac{\partial w_{P_{i}}}{\partial x_j}\, dx
\end{eqnarray}
According to (\ref{prodotto}), the coefficient matrix is nonsingular since it is dominated by its diagonal. By standard elliptic regularity, $v\in L^{\infty}(\mathbb R^N)\cap H^2(\mathbb R^N)$. Furthermore, using the $C^{1, \sigma}$ regularity of $f$ and the exponential decay of $w$
$$\|f'(w_{{\bf P}})v-2 \phi[w_{{\bf P}}v]w_{\bf P}-h-\sum_{i, j}c_{i, j}Z_{P_{i}, j}\|_{*, {\bf P}}<\infty$$ hence Lemma \ref{PDEest} implies $\|v\|_{*, {\bf P}}<\infty$, consequently $(v, c_{i, j})$ solves (\ref{problin}).\\ Thus it remains to solve (\ref{sy}). According to Riesz's representation theorem, take $\mathcal{K}_{\bf P}(v)$, $\bar{h} \in W$ such that
\begin{eqnarray*}
(\mathcal{K}_{\bf P}(v), \psi)_{\epsilon}&=&-\langle f'(w_{\bf P})v, \psi \rangle +\langle\phi[{w_{\bf P}^2}]v+2\phi[w_{\bf P}v]w_{\bf P}, \psi \rangle\,\,\, \forall\,\, \psi\in W.\\
 (\bar{h}, \psi)_{\epsilon}&=&-\langle h, \psi \rangle\,\,\, \forall\,\,\ \psi \in W.
\end{eqnarray*}
Then the problem (\ref{sy}) consists in finding $v\in W$ such that
\begin{equation}\label{syeq1}
v+\mathcal{K}_{\bf P}(v)=\bar{h}.
\end{equation}
It is easy to prove that $\mathcal{K}_{\bf P}$ is a linear compact operator from $W$ to $W$.\\ Using Fredholm's alternatives, (\ref{syeq1}) has a unique solution for each $\bar{h}$, if and only if (\ref{syeq1}) has a unique solution for $\bar{h}=0$. Let $v\in W$ be a solution of $v+\mathcal{K}_{\bf P}(v)=0$ then $v$ solves the system (\ref{problin}) with $h=0$ for some $c_{i, j}\in \mathbb R$. Lemma \ref{Estnorm} implies $v\equiv 0$.
\end{proof}

\

\section{Finite dimensional reduction }\label{Section finite reduction}

This section is devoted to solve the following nonlinear system with the unknowns $(\psi, c_{i, j})\in H^2_{*, {\bf P}}\times \mathbb R^{N\ell}$:
\begin{equation}\label{nsy}
\left\{
\begin{array}{lr}
S_{\epsilon}[w_{\bf P}+\psi]=\sum_{i, j}c_{i, j}Z_{P_{i}, j},\\
\psi\in H^2_{*, {\bf P}},\,\, \langle \psi, Z_{P_{i}, j}\rangle=0,\,\,\ i=1, \ldots, \ell,\,\, j=1, \ldots, N.
\end{array}
\right.
\end{equation}
where ${\bf P}\in \bar{\Gamma}_{\epsilon}$ and $$S_{\epsilon}[v]=\epsilon^2\Delta v-v+f(v)-\phi[{v}^2]v.$$
We prove the following result.
\begin{lemma}\label{fdr}
Fix $\tau=\beta^4(1+\sigma)$. Provided $\epsilon>0$ sufficiently small, for every ${\bf P}\in \bar{\Gamma}_{\epsilon}$ there is a unique pair $(\psi_{\bf P}, c_{i, j}({\bf P}))\in H^2_{*, {\bf{ P}}}\times \mathbb R^{N\ell}$ which solve (\ref{nsy}). Moreover
\begin{equation}
\|\psi_{\bf P}\|_{*, {\bf{ P}}}<\epsilon^{\tau}; \quad (\psi_{\bf P}, \psi_{\bf P})_{\epsilon}\leq \epsilon^{N+2\tau}
\end{equation}
and the maps ${\bf P}\in \bar{\Gamma}_{\epsilon}\longmapsto \psi_{\bf P}\in H^1(\mathbb R^N)$ and ${\bf P}\longmapsto c_{i, j}(\bf P)\in \mathbb R$ are continuous.
\end{lemma}
\begin{proof}
We note that $$S_{\epsilon}[w_{\bf P}+\psi]=S_{\epsilon}[w_{\bf P}]+\mathcal{L}_{\bf P}[\psi]+R[\psi]$$ with $$R[\psi]= \left[f(w_{\bf P}+\psi)-f'(w_{\bf P})\psi-f(w_{\bf P})\right]-\left[\phi[{\psi}^2]w_{\bf P}+\phi[{\psi}^2]\psi+2\phi[w_{\bf P}\psi]\psi\right]$$ Hence (\ref{nsy}) can be written as
\begin{equation}\label{nsy1}
\left\{
\begin{array}{lr}
\mathcal{L}_{\bf P}[\psi]= \sum_{i, j}c_{i, j}Z_{P_{i}, j}-S_{\epsilon}[w_{\bf P}]-R[\psi],\\\\
\psi\in H^2_{*, {\bf P}},\,\, \langle \psi, Z_{P_{i}, j}\rangle=0,\,\,\ i=1, \ldots, \ell,\,\, j=1, \ldots, N.
\end{array}
\right.
\end{equation}
that is (\ref{problin}) with $$h=-S_{\epsilon}[w_{\bf P}]-R[\psi].$$ Let us consider the metric space $${\mathcal B}=\left\{\psi\in C(\mathbb R^N)\,\,:\,\, \|\psi\|_{*, {\bf P}}\leq \epsilon^{\tau}\right\}$$ endowed with the norm $\|\cdot\|_{*, {\bf P}}$.  For all $\psi_1, \psi_2\in {\mathcal B}$ $$\|R[\psi_1]-R[\psi_2]\|_{*, {\bf P}}\leq C\epsilon^{\sigma\tau}\|\psi_1-\psi_2\|_{*, {\bf P}}$$ and for all $\psi\in {\mathcal B}$
\begin{equation}\label{5.18}
\|R[\psi]\|_{*, {\bf P}}\leq C\epsilon^{\tau}\|\psi\|_{*, {\bf P}}\leq \epsilon^{(1+\sigma)\tau}.
 \end{equation}
 Moreover, by Lemma \ref{LemmaStimaErroreS} we have that
 \begin{equation}\label{5.19}
 \|S_{\epsilon}[\psi]\|_{*, {\bf P}}\leq \epsilon^{\beta^2(\beta^2+\sigma)}.
  \end{equation}
  Thus, for all $\psi\in {\mathcal B}$, let $(\mathcal{A}(\psi), c_{i, j})$ the unique solution of (\ref{problin}) with $$h=-S[w_{\bf P}]-R[\psi]$$ Then we claim that $\mathcal{A}$ maps $\mathcal B$ into $\mathcal B$ and $\mathcal{A}$ is a contraction. By lemma \ref{EsUn} and the choice of $\tau$ $$\|\mathcal{A}(\psi)\|_{*, {\bf P}}\leq C\|-S_{\epsilon}[w_{\bf P}]-R(\psi)\|_{*, {\bf P}}\leq C\left(\epsilon^{\beta^2(\beta^2+\sigma)}+\epsilon^{(1+\sigma)\tau}\right)\leq C\epsilon^{\tau}$$ for $\epsilon$ sufficiently small and so $\mathcal{A}(\psi)\in {\mathcal B}$. Moreover $\mathcal{A}(\psi_1)-\mathcal{A}(\psi_2)$ solves (\ref{problin}) with $h=-R[\psi_1]+R[\psi_2]$. Then by Lemma \ref{EsUn} $$\|\mathcal{A}(\psi_1)-\mathcal{A}(\psi_2)\|_{*, {\bf P}}\leq C\|R[\psi_1]-R[\psi_2]\|_{*, {\bf P}}\leq \epsilon^{\tau}\|\psi_1-\psi_2\|_{*, {\bf P}}.$$ and so, for $\epsilon$ small, $\mathcal{A}$ is a contraction. Thus, by applying the contraction mapping theorem we conclude. It remains to prove the $H^1-$ norm estimate of $\psi$. By multiplying (\ref{nsy1}) by $\psi_{\bf P}$ and integrating by parts we obtain
\begin{eqnarray}\label{5.20}
\nonumber
(\psi_{\bf P}, \psi_{\bf P})_{\epsilon}&=&\int_{\mathbb R^N}f'(w_{\bf P})\psi_{\bf P}^2\, dx -\int_{\mathbb R^N}\phi[w_{\bf P}^2]\psi_{\bf P}^2\, dx-2\int_{\mathbb R^N}\phi[w_{\bf P}\psi_{{\bf P}}]w_{\bf P}\psi_{\bf P}\, dx+\\
&&+\langle S_{\epsilon}[w_{\bf P}], \psi_{\bf P}\rangle +\langle R[\psi_{\bf P}], \psi_{\bf P}\rangle
\end{eqnarray}
By using the fact that $\psi_{\bf P}\in {\mathcal B}$, the estimates (\ref{5.19}) and (\ref{5.20}) and by making a change of variable we immediately get $$(\psi_{\bf P}, \psi_{\bf P})_{\epsilon}\leq C\epsilon^{N+2\tau}.$$ Then the family $\left\{\psi_{\bf P}\,\, :\,\, {\bf P}\in \bar{\Gamma}_{\epsilon}\right\}$ is bounded in $H^1$. Now fix $\epsilon>0$ and consider $\{{\bf P}_n\}\subset \bar{\Gamma}_{\epsilon}$ such that ${\bf P}_n \rightarrow \bar{P}\in \bar{\Gamma}_{\epsilon}$. Up to a subsequence, $\psi_{{\bf P}_n}\rightharpoonup {\bf \bar{\psi}}$ weakly in $H^1$; on the other hand, choosing $(m, q)$ such that, up to a subsequence $|c_{m, q}({\bf P}_n)|\geq |c_{i, j}({\bf P}_n)|$ for every $(i, j)$ and $n$, by using (\ref{prodotto}) we have
\begin{eqnarray*}
\left(\psi_{{\bf P}_n}, \frac{\partial w_{{\bf P}^n_m}}{\partial x_q}\right)_{\epsilon}&=&\int_{\mathbb R^N}f'(w_{{\bf P}_n})\psi_{{\bf P}_n}\frac{\partial w_{{\bf P}^n_m}}{\partial x_q}\, dx -\int_{\mathbb R^N}\phi[{w_{{\bf P}_n}^2}]\psi_{{\bf P}_n}\frac{\partial w_{{\bf P}^n_m}}{\partial x_q}\, dx\\
&&-2\int_{\mathbb R^N}\phi\left[w_{{\bf P}_n}\frac{\partial w_{{\bf P}^n_m}}{\partial x_q}\right]w_{{\bf P}_n}\frac{\partial w_{{\bf P}^n_m}}{\partial x_q}\, dx+\langle S_{\epsilon}[w_{{\bf P}_n}], \frac{\partial w_{{\bf P}^n_m}}{\partial x_q}\rangle \\
&&+\langle R[\psi_{{\bf P}_n}], \frac{\partial w_{{\bf P}^n_m}}{\partial x_q}\rangle-c_{m, q}({\bf P}_n)\left(\epsilon^{N-2}\left\|\frac{\partial w}{\partial x_1}\right\|^2+o\left(\epsilon^{N-2}\right)\right)
\end{eqnarray*}
by which we deduce that the sequence $\{c_{i, j}({\bf P}_n)\}$ is bounded too for every $(i, j)$. Assume, without loss of generality $c_{i, j}({\bf P}_n)\rightarrow \bar{c}_{i, j}$. Then $(\bar{\psi}, \bar{c}_{i, j})$ solves the equation
$$\mathcal{L}_{\bar{\bf P}}(\bar{\psi})=-S_{\epsilon}[w_{\bar{\bf P}}]-R[\bar{\psi}]+\sum_{i, j}\bar{c}_{i, j}Z_{\bar{P}_i, j}, \quad \langle \bar{\psi}, Z_{\bar{P}_i, j}\rangle=0,\quad \|\bar{\psi}\|_{*, {\bf P}}\leq \epsilon^{\tau}.$$ Hence, from uniqueness, it follows $\bar{\psi}=\psi_{\bar{\bf P}}$ and $\bar{c}_{i, j}=c_{i, j}(\bar{\bf P})$, By (\ref{5.20}) we get
\begin{eqnarray*}
\|\psi_{{\bf P}_n}\|^2 &\rightarrow&
\int_{\mathbb R^N}f'(w_{\bar{\bf P}})\psi_{\bar{\bf P}}^2\, dx -\int_{\mathbb R^N}\phi[{w_{\bar{\bf P}}^2}]\psi_{\bar{\bf P}}^2\, dx-2\int_{\mathbb R^N}\phi[w_{\bar{\bf P}}\psi_{\bar{\bf P}}]w_{\bar{\bf P}}\psi_{\bar{\bf P}}\, dx+\\
&&+\langle S_{\epsilon}[w_{\bar{\bf P}}], \psi_{\bar{\bf P}}\rangle +\langle R[\psi_{\bar{\bf P}}], \psi_{\bar{\bf P}}\rangle=\|\bar{\psi}\|^2,
\end{eqnarray*}
hence we deduce $\psi_{\bf P_n}\rightarrow \psi_{\bar{\bf P}}$ in $H^1$.
\end{proof}
\begin{lemma}
For $\epsilon>0$ sufficiently small the map ${\bf P}\in \bar{\Gamma}_{\epsilon}\longmapsto \psi_{\bf P}\in H^1$ constructed in Lemma \ref{fdr} is $C^1$.
\end{lemma}
\begin{proof}
To prove that the map ${\bf P}\in \Gamma_{\epsilon}\rightarrow \psi_{\bf P}\in H^1$ is $C^1$ consider the following map $T: \Gamma_{\epsilon}\times H^1(\mathbb R^N)\times\mathbb R^{N\ell}\rightarrow H^1(\mathbb R^N)\times \mathbb R^{N\ell}:$
\begin{equation}
T({\bf P}, \psi_{\bf P}, c_{i, j})=
\left(
\begin{array}{lr}
(\epsilon^2\Delta-1)^{-1}\left(S_{\epsilon}[w_{\bf P}+\psi_{\bf P}]+\sum_{i, j}c_{i, j}\frac{\partial w_{P_i}}{\partial x_j}\right)\\\\
\left(\psi_{\bf P}, \frac{\partial w_{P_i}}{\partial x_j}\right)_{\epsilon}
\end{array}
\right)
\end{equation}
where $v=(\epsilon^2\Delta-1)^{-1}(h)$ is defined as the unique solution $u\in H^1$ of $\epsilon^2\Delta v-v=h$. Since $-\epsilon^2\Delta\frac{\partial w_{P_i}}{\partial x_j}-\frac{\partial w_{P_i}}{\partial x_j}=-Z_{P_i, j}$ it is immediate that $(\psi, c_{i, j})$ solves the system (\ref{nsy}) if and only if $T({\bf P}, \psi, c_{i, j})=0$. The thesis will follow by applying the Implicit Function Theorem (see \cite{DW}).
\end{proof}

\

\section{Reduced energy functional}\label{sectionReduced}
For $\epsilon>0$ sufficiently small we define the reduced functional
$M_{\epsilon}:\bar{\Gamma}_{\epsilon}\rightarrow\mathbb R$ \begin{equation}
\label{funzioRidot}
 M_{\epsilon}[{\mathbf P}]:=\epsilon^{-N} J_{\epsilon}[w_{{\bf P}}+\psi_{{\bf P}}]-\ell I[w]-\epsilon^2 C_1,\end{equation}
where $\psi_{{\bf P}}$ has been constructed in Lemma \ref{fdr} and $C_1$ is given by Proposition \ref{4.1}.\\

Next proposition contains the key expansion of $M_{\epsilon}$

\begin{proposition}\label{ridotto generale}
For $\epsilon>0$ sufficiently small the following holds:
\[M_{\epsilon}[{\bf P}]= - (\gamma_0+o(1))\sum_{i\neq j}\lambda_i\lambda_jw\left(\frac{P_i-P_j}{\epsilon}\right) \,+\, \epsilon^{2}(C_2+o(1))\sum_{i\neq j}\frac{1}{|\frac{P_i-P_j}{\epsilon}|^{N-2}}
\, +\,  O(\epsilon^{2\tau}),\]
uniformly for ${\bf P}\in \bar{\Gamma}_{\epsilon},$ where $\tau= \beta^4(1+\sigma)$ is given by Lemma \ref{fdr} and $\gamma_0,C_2$ are the constants in Proposition \ref{4.1}.
\end{proposition}

\begin{proof} An easy computation gives
\begin{eqnarray*}
&&\hskip-0.7cm  J_{\epsilon}(w_{{\bf P}}+\psi_{{\bf P}}) 
= J_{\epsilon}(w_{{\bf P}}) - \int_{\mathbb R^N} S_{\epsilon}[w_{{\bf P}}]\psi_{{\bf P}}\, dx + \frac{1}{2}(\psi_{{\bf P}},\psi_{{\bf P}})_{\epsilon} - \int_{\mathbb R^N} \left(F(w_{{\bf P}}+\psi_{{\bf P}})-F(w_{{\bf P}})-f(w_{{\bf P}})\psi_{{\bf P}}\right)\, dx\\
&& + \int_{\mathbb R^N} \psi_{{\bf P}}^2\left(\frac{1}{4} \phi[\psi_{{\bf P}}^2]+\frac{1}{2}\phi[w_{{\bf P}}^2](x)+\phi[w_{{\bf P}}\psi_{{\bf P}}]\right)\, dx+ \int_{\mathbb R^N}\!\!\int_{\mathbb R^N} \frac{w_{{\bf P}}(x)\psi_{{\bf P}}(x)w_{{\bf P}}(y)\psi_{{\bf P}}(y)}{|x-y|^{N-2}}\, dx\, dy\\
&& + \int_{\mathbb R^N}\!\!\int_{\mathbb R^N} \frac{w_{{\bf P}}^2(x)w_{{\bf P}}(y)\psi_{{\bf P}}(y)}{|x-y|^{N-2}}\, dx\, dy.\\
\end{eqnarray*}
By Lemma \ref{LemmaStimaErroreS} we have
\begin{eqnarray*}&& \hskip-0.7cm \left|\int_{\mathbb R^N} S_{\epsilon}[w_{{\bf P}}]\psi_{{\bf P}}\, dx\right| \leq C \epsilon ^{\beta^2(\beta^2+\sigma)} \|\psi_{{\bf P}}\|_{\infty}\sum_{i=1}^{\ell}\int_{\mathbb R^N} w_{P_i}^{1-\beta^2}\, dx
\leq C \epsilon ^{\tau} \|\psi_{{\bf P}}\|_{\infty}\sum_{i=1}^{\ell}\int_{\mathbb R^N} w_{P_i}^{1-\beta^2}\, dx = \tilde C \epsilon ^{N+\tau} \|\psi_{{\bf P}}\|_{\infty}.\end{eqnarray*}
Moreover, since $\left|F(w_{{\bf P}}+\psi_{{\bf P}})-F(w_{{\bf P}})-f(w_{{\bf P}})\psi_{{\bf P}}\right|\leq C|\psi_{{\bf P}}|^2,$ one can estimate $$\left|\int_{\mathbb R^N} \left(F(w_{{\bf P}}+\psi_{{\bf P}})-F(w_{{\bf P}})-f(w_{{\bf P}})\psi_{{\bf P}}\right|)\, dx\right|\leq C\|\psi_{{\bf P}}\|^2.$$
It's also easy to see that
$$\int_{\mathbb R^N} \psi_{{\bf P}}^2\left(\frac{1}{4} \phi[\psi_{{\bf P}}^2]+\frac{1}{2}\phi[w_{{\bf P}}^2](x)+\phi[w_{{\bf P}}\psi_{{\bf P}}]\right)\, dx+  \int_{\mathbb R^N}\!\!\int_{\mathbb R^N} \frac{w_{{\bf P}}(x)\psi_{{\bf P}}(x)w_{{\bf P}}(y)\psi_{{\bf P}}(y)}{|x-y|^{N-2}}\,dx\,dy
\leq  C\|\psi_{{\bf P}}\|^2.$$
Last, similarly as in the proof of \eqref{stimaEnergiaPoisson}  (using now \eqref{laSeconda} instead of \eqref{laPrima}), one has

\begin{eqnarray*}\int_{\mathbb R^N}\!\!\int_{\mathbb R^N} \frac{w_{{\bf P}}^2(x)w_{{\bf P}}(y)\psi_{{\bf P}}(y)}{|x-y|^{N-2}}\,dx\,dy&\leq& \|\psi_{\bf P}\|_{\infty}
\int_{\mathbb R^N}\!\!\int_{\mathbb R^N} \frac{w_{{\bf P}}^2(x)w_{{\bf P}}(y)}{|x-y|^{N-2}}\,dx\,dy\leq C\epsilon^{2+N} \|\psi_{\bf P}\|_{\infty}
\sum_{i\neq j}\frac{1}{|\frac{P_j-P_i}{\epsilon}|^{N-2}}\\
&\leq& C\epsilon^{2+N} \|\psi_{\bf P}\|_{\infty}
\sum_{i\neq j}\frac{1}{\log{(\frac{1}{\epsilon^{2\beta^2}})}^{N-2}}\leq C\epsilon^{2+N} \|\psi_{\bf P}\|_{\infty}.\end{eqnarray*}
Hence by Lemma \ref{fdr} 
(observe that by our assumptions $2+\tau>2\tau$) one obtains
\[J_{\epsilon}(w_{{\bf P}}+\psi_{{\bf P}}) = J_{\epsilon}(w_{{\bf P}})+ O(\epsilon^{N+2\tau}),
\]
and the thesis easily follows from Proposition \ref{4.1}.
\end{proof}

%
%
%
%
%
%
%

We recall the following result (whose proof can be found for instance in \cite{DW}) that will be useful in the next sections in order to find a critical point (a maximum) of $M_{\epsilon}$ under symmetry assumptions
\begin{lemma}\label{Lemma Massimo Ridotto}
Fix a positive constant $C>0$ and consider the function
$$\alpha_{\epsilon, C}(\rho):=-\gamma_0w(\rho)+C\epsilon^2\frac{1}{\rho^{N-2}},\ \ \ \rho\geq \beta^2\log{\frac{1}{\epsilon^2}},$$
where $\gamma_0$ is the positive constants introduced in Proposition \ref{4.1}. Then for $\epsilon>0$ small enough, $\alpha_{\epsilon,C}$ has a unique maximum point $\rho_{\epsilon}.$ Moreover we have
$$\rho_{\epsilon}=\log{\frac{1}{\epsilon^2}}+\frac{N-1}{2}\log{\log{\frac{1}{\epsilon^2}}}+o\left(\log{\log{\frac{1}{\epsilon^2}}}\right)$$
and
$$\alpha_{\epsilon,C}(\rho_{\epsilon})=C\epsilon^2\frac{1}{\left( \log{\frac{1}{\epsilon^2}}\right)^{N-2}}(1+o(1)).$$
\end{lemma}

\

\

\section{Proof of Theorem \ref{theorem poligono piano} }
For every $x\in\mathbb R^N$ we set $x=(x_1,\ldots ,x_N)=(x_1,x_2,x')=(z,x'),$ where $z\in \mathbb C.$ \\In this section we prove the existence of a cluster solution to \eqref{eqScalare} having a positive bump in $0$ and $k$ negative bumps at the vertices of a regular polygon centered in $0$.

Precisely in this case  $\ell=k+1$ and and we look for a solution of the form \eqref{generalansatz} making the following {\it ansatz}  
\begin{equation}\label{ansatzSimmetrico1}
v=w_{\bf P}+\psi_{\epsilon} 
\end{equation}
where
\[w_{\bf P}=w_0-\sum_{i=1}^kw_{P_i}\]
and $P_i:=rQ_i\in\mathbb R^N,$
$Q_i:=(e^{2\pi \sqrt{-1}(i-1)/k}, {\bf 0})=\left(\cos{\frac{2\pi(i-1)}{k}},\sin{\frac{2\pi(i-1)}{k}}, {\bf 0}\right)\in\mathbb R^N,$ $i=1,\ldots,k,$ $0=P_{k+1}=(0,0,{\bf 0})$
and  $r\in R_{\epsilon}:=\left\{r >0\ : \,\, \frac{\beta^2}{\beta_k}\epsilon\log{\frac{1}{\epsilon^2}}<r <\frac{1}{2}\epsilon \left(\log{\frac{1}{\epsilon^2}}\right)^2
\right\},$ for   $\beta_k:=2\sin{\frac{\pi}{k}}$.\\

Observe that with this choice of $r$ the point ${\bf P}_r:=(P_1,\ldots,P_k,P_{k+1})\in \Gamma_{\epsilon}(\subset \mathbb R^{(k+1)N})$, where $\Gamma_{\epsilon}$ is the configurations set introduced in \eqref{gammeepsilon}. Indeed, by the definition of $\Gamma_{\epsilon}$, ${\bf P}_r\in \Gamma_{\epsilon}$ if and only if
\[\left\{
\begin{array}{lr}
 \beta^2\epsilon\log{\frac{1}{\epsilon^2}}<r <\epsilon \left(\log{\frac{1}{\epsilon^2}}\right)^2\\
\beta^2\epsilon\log{\frac{1}{\epsilon^2}}<r|Q_i-Q_j| <\epsilon \left(\log{\frac{1}{\epsilon^2}}\right)^2,\quad i\neq j
\end{array}
\right.\] and by the assumption $k\geq 7$, it follows that $\beta_k<1$  and moreover it is not difficult to see that
 $\beta_k\leq |Q_i-Q_j|\leq 2$ for $i\neq j$. As a consequence we are in the good framework to obtain all the results in Sects. \ref{section linearized} and \ref{Section finite reduction}. 
\\

In addition here we look for a solution $v$ satisfying also the following symmetry properties
\[v(x_1,x_2,x')=v(z,|x'|)=v(ze^{2\pi\sqrt{-1}/k},|x'|)\]
\[v(x_1,x_2,x')=v(x_1,-x_2,x').\]
This translates into restricting to work into the following Sobolev space of symmetric functions
\[X:=\{v\in H^1(\mathbb R^N): v(x_1,x_2,x')=v(z,|x'|)=v(ze^{2\pi\sqrt{-1}/k},|x'|),\ v(x_1,x_2,x')=v(x_1,-x_2,x')\}.\]
Hence, for every $r\in R_{\epsilon},$ we set
\begin{equation}\label{spazioSimmetrico1}H^2_{\ast,r,s}:=H^2_{\ast, {\bf P}_{r}}\cap X, \ \ C_{\ast,r,s}:=C_{\ast, {\bf P}_{r}}\cap X,\end{equation}
and proceeding as in Sects. \ref{section linearized} and \ref{Section finite reduction}, we find  for $\epsilon$ small enough
a unique solution $(\psi_{{\bf P}_r}, c_{ij}({\bf P}_r))\in H^2_{\ast, r, s}\times \mathbb R^{N\ell}$ to problem \eqref{nsy} (see Lemma \ref{fdr}).
\\

Following \cite{DelPinoFelmerMusso} it is possible to show that, restricting to work on $X$ and with the symmetric choice of the point ${\bf P_r}$ that we have done, the unknowns $(c_{i, j})_{\substack{i=1,\dots, \ell\\j=1,\dots,N}}\in  \mathbb R$ in problem \eqref{nsy} reduces  only to one unknown $c_{1,1}$, precisely we prove the following result
\begin{lemma}\label{LemmaAusiliariaSimmetrico1}
Fix $\tau=\beta^4(1+\sigma)$. Provided $\epsilon>0$ is sufficiently small, for every $r\in R_{\epsilon}$ there is a unique pair $(\psi_{\bf P_r}, c_{1, 1}({\bf P_r}))\in H^2_{*,r,s}\times \mathbb R$ which solve 
\begin{equation} \label{ausiliariaRidottaPersimmetria}
\left\{
\begin{array}{lr}
S_{\epsilon}[w_{\bf P_r}+\psi]=c_{1, 1}\sum_{i=1,\cdots, k}\left[ \cos\left({\frac{2\pi}{k}(i-1)}\right) Z_{P_{i}, 1}+\sin\left({\frac{2\pi}{k}(i-1)}\right) Z_{P_{i}, 2}\right],\\
\psi\in H^2_{*, r,s},\,\, \langle \psi, Z_{{P}_{i}, j}\rangle=0,\,\,\ i=1, \ldots, \ell,\,\, j=1, \ldots, N.
\end{array}
\right..
\end{equation}
Moreover
\[
\|\psi_{\bf P_r}\|_{*,{\bf P}_r}<\epsilon^{\tau}; \quad (\psi_{\bf P_r}, \psi_{\bf P_r})_{\epsilon}\leq \epsilon^{N+2\tau}
\]
the map ${r}\in \bar{R}_{\epsilon}\longmapsto \psi_{\bf P_r}\in H^1(\mathbb R^N)$ is $C^1$ and the map ${r}\longmapsto c_{1, 1}(\bf P_r)\in \mathbb R$ is continuous.
\end{lemma}
\begin{proof}
We postpone the proof of this result to the Appendix.
\end{proof}
$\;$\\
Let us set, for $\epsilon>0$ sufficiently small, the one-variable function $\widetilde M_{\epsilon}:\bar{R_{\epsilon}}\rightarrow \mathbb R$
\[\widetilde M_{\epsilon}[r]:= M_{\epsilon}[{\bf P}_r],\]
where $M_{\epsilon}$ is the reduced functional defined in \eqref{funzioRidot}.
\\
\\
To conclude the proof it is sufficient to find for $\epsilon$ small, a critical point $r$ of the  function $\widetilde M_{\epsilon}$. Indeed  the following holds
\begin{lemma} \label{lemmarelazione punti critici ridotto e funzionale} Let ${\bar r}\in R_{\epsilon}$ be an interior maximum point for $\widetilde M_{\epsilon}.$ Then, for $\epsilon>0$ sufficiently small, the corresponding function $v_{\epsilon}:=w_{{\bf P}_{\bar r}}+\psi_{{\bf P}_{\bar r}}$ is a critical point of $J_{\epsilon}$, namely a solution to \eqref{eqScalare}.
\end{lemma}
\begin{proof}
The proof consists in showing that if ${\bar r}\in R_{\epsilon}$ is an interior maximum point for $\widetilde M_{\epsilon},$ then
%
%
%
%
%
%
%
\begin{equation}
\label{c11=zero} 
c_{1,1}({\bf P}_{\bar r})=0.\end{equation}
Indeed it is clear from \eqref{ausiliariaRidottaPersimmetria} that for  a point $\bf P$ solving \eqref{c11=zero}, the corresponding function $v=w_{\bf P}+\psi_{\bf P}$ is a critical point of $J_{\epsilon}$ on $X$.
\\
Hence, if we denote by $G$ the group of the rotation matrix in $\mathbb R^{N-2},$ and
for every $i\in \mathbb N$ and $g\in G$ we define \[T_{i,g}:\mathbb R^N\rightarrow\mathbb R, \quad T_{i,g}(x)=T_{i,g}(z,x'):=(ze^{2\pi i\sqrt{-1}/k},gx'),\]
\[\widetilde T_{2,g}:\mathbb R^N\rightarrow\mathbb R, \quad \widetilde T_{2,g}(x)=\widetilde T_{2,g}(x_1,x_2,x'):=(x_1,-x_2,gx'),\]
then, by Lemma \ref{invact} the functional $J_{\epsilon}$ is invariant under the action of the group $\{T_{i,g}, \widetilde T_{2,g}:i\in\mathbb N, g\in G\}.$  Moreover $X=\{u\in H^1(\mathbb R^N): u(T_{i,g}(x))=u(x), u(\widetilde T_{2,g}(x))=u(x)\}.$ So  the  principle of symmetric criticability ensures that $v$ is also a critical point of $J_{\epsilon}$ and, consequently, a solution of \eqref{eqScalare}.
\\
\\
In the following we show that  \eqref{c11=zero} holds.
Since $\bar r\in R_{\epsilon}$ is an interior maximum point for $\widetilde M_{\epsilon}$, then in particular
\begin{equation}
\label{derinulla}
\frac{\partial}{\partial r}\widetilde M_{\epsilon}(r)\big|_{r=\bar r}=0.
\end{equation}
Using the $C^1$ regularity of the map $r\mapsto\psi_{{\bf P}_r}$, \eqref{derinulla} may be rewritten as
\[
\int_{\mathbb R^N} S_{\epsilon}[w_{{\bf P}_{\bar r}}+\psi_{{\bf P}_{\bar r}}]\frac{\partial}{\partial r}\left( w_{{\bf P}_r}+\psi_{{\bf P}_r}\right)\big|_{r=\bar r} \ dx =0,
\]
which is equivalent by  \eqref{ausiliariaRidottaPersimmetria} to
\begin{equation}
\label{dernullacontinuo}
c_{1, 1}\int_{\mathbb R^N} \sum_{i=1}^k\left[ \cos\left({\frac{2\pi}{k}(i-1)}\right) Z_{P_{i}, 1}+\sin\left({\frac{2\pi}{k}(i-1)}\right) Z_{P_{i}, 2}\right]
\frac{\partial}{\partial r}\left( w_{{\bf P}_r}+\psi_{{\bf P}_r}\right)\big|_{r=\bar r} \ dx =0.
\end{equation}
Now
\begin{eqnarray}\label{eqderivataw}
\frac{\partial}{\partial r} w_{{\bf P}_r} &=& \frac{\partial}{\partial r}\left(\sum_{i=1}^{k+1} w_{p_i}\right) = -\sum_{i=1}^{k}\left({Q_{i}}_1\frac{\partial w_{P_i}}{\partial x_1}+{Q_{i}}_2\frac{\partial w_{P_i}}{\partial x_2} \right)\nonumber\\
&= &-\sum_{i=1}^{k}\left[\left(\cos \frac{2\pi}{k}(i-1)\right)\frac{\partial w_{P_i}}{\partial x_1}+\left(\sin \frac{2\pi}{k}(i-1)\right)\frac{\partial w_{P_i}}{\partial x_2} \right]
\end{eqnarray}
Moreover, since $\langle \psi_{\bf P_r}, Z_{{P}_{i}, j}\rangle=0$, then
\begin{equation}
\label{equainterm}
\int_{\mathbb R^N}Z_{P_i,j}\frac{\partial \psi_{\bf P_r}}{\partial r} dx= 
-\int_{\mathbb R^N}\frac{\partial Z_{P_i,j}}{\partial r}\psi_{\bf P_r} =
-\left( {Q_{i}}_1\int_{\mathbb R^N}\frac{\partial Z_{P_i,j}}{\partial x_1}\psi_{\bf P_r}dx\  +\ {Q_{i}}_2 \int_{\mathbb R^N} \frac{\partial Z_{P_i,j}}{\partial x_2}\psi_{\bf P_r}dx  \right),
\end{equation} 
and for $m=1,2$, by H\"older inequality
\[
\int_{\mathbb R^N}
\frac{\partial Z_{P_i,j}}{\partial x_m}\psi_{\bf P_r}dx=
\int_{\mathbb R^N}\left(
\epsilon^2 \nabla\left(\frac{\partial^2 w_{P_i}}{\partial x_m\partial x_j}\right)\nabla\psi_{\bf P_r}+
\frac{\partial^2 w_{P_i}}{\partial x_m\partial x_j}\psi_{\bf P_r}\right) dx\leq \left\| \frac{\partial^2 w_{P_i}}{\partial x_m\partial x_j}\right\|\|\psi_{\bf P_r}  \|. 
\]
By a change of variable it is easy to see that
\[\left\| \frac{\partial^2 w_{P_i}}{\partial x_m\partial x_j}\right\|=\int_{\mathbb{R}^N}\left( \epsilon^2\left|\nabla \frac{\partial^2 w_{P_i}}{\partial x_m\partial x_j}\right |^2+\left(\frac{\partial^2 w_{P_i}}{\partial x_m\partial x_j}\right)^2\right) dx= \epsilon^{N-4} \left\| \frac{\partial^2 w}{\partial x_m\partial x_j}\right\|_{H^1(\mathbb R^N)} \]
ad also by Lemma \ref{LemmaAusiliariaSimmetrico1}
\[\|\psi_{\bf P_r}  \|\leq \epsilon^{\frac{N+2\tau}{2}}\]
hence substituting into \eqref{equainterm} we get
\begin{equation}\label{eqderivPsi}
\int_{\mathbb R^N}Z_{P_i,j}\frac{\partial \psi_{\bf P_r}}{\partial r} dx
=O(\epsilon^{\frac{3N}{2}+\tau-4})\end{equation}
Hence substituting \eqref{eqderivataw} and \eqref{eqderivPsi} into \eqref{dernullacontinuo},  and also using \eqref{prodotto}, we get
\[c_{1,1}\left[k\left\|\frac{\partial w}{\partial x_1}\right\|^2 + o(1)+ O(\epsilon^{\frac{N}{2}+\tau-2})\right]=0\]
from which $c_{1,1}=0$.
\end{proof}

The remaining part of the section is then devoted to find an interior maximum point of the reduced functional $\widetilde M_{\epsilon}$.

Let us observe that, thanks to Proposition \ref{ridotto generale} and to the assumption $k\geq 7,$ it reduces to the following
\begin{proposition}
For $\epsilon>0$ sufficiently small
\begin{equation}\label{ridottoSimmetricoPoligono}
\widetilde M_{\epsilon}[r]= (2k+o(1)) \alpha_{\epsilon, C_k}\left(\frac{r}{\epsilon}\beta_k\right)+ O(\epsilon^{2\tau})
\end{equation}
uniformly for  $r>0$ such that $r\in \bar R_{\epsilon},$ where $\beta_k:=2\sin{\frac{\pi}{k}},$ ${C_k}$ is a positive constant and $\alpha_{\epsilon, C_k}$ is the function defined in Lemma \ref{Lemma Massimo Ridotto}.
\end{proposition}

\begin{proof}
For $r\in \bar R_{\epsilon}$ the point ${\bf P}_r\in\bar \Gamma_{\epsilon}$, and the reduced functional becomes
\begin{eqnarray*}
\widetilde M_{\epsilon}[r]&=& - (\gamma_0+o(1))\left( -2k w\left(\frac{r}{\epsilon}\right)+
\sum_{i\neq j}w\left(\frac{r}{\epsilon}|Q_i-Q_j|\right)\right)\\ && +\epsilon^{2}(C_2+o(1))\frac{1}{\left|\frac{r}{\epsilon}\right|^{N-2}}\left(2k+\sum_{i\neq j}\frac{1}{|Q_i-Q_j|^{N-2}}\right)
+ O(\epsilon^{2\tau})\\
&=& - k(\gamma_0+o(1))\left( -2 w\left(\frac{r}{\epsilon}\right)+
\sum_{i=2}^kw\left(\frac{r}{\epsilon}|Q_1-Q_i|\right)\right)\\
&&+\epsilon^{2}k(C_2+o(1))\frac{1}{\left|\frac{r}{\epsilon}\right|^{N-2}}\left(2+\sum_{i=2}^k\frac{1}{|Q_1-Q_i|^{N-2}}\right)
+ O(\epsilon^{2\tau})\\
& =&
 - k(\gamma_0+o(1))\left( -2 w\left(\frac{r}{\epsilon}\right)+ 2 w(\frac{r}{\epsilon}\beta_k)+
\sum_{i=3}^{k-1}w\left(\frac{r}{\epsilon}\beta_k^i\right)\right)\\
&&+\epsilon^{2}k(C_2+o(1))\frac{1}{\left|\frac{r}{\epsilon}\right|^{N-2}}\left(2+2\frac{1}{\beta_k^{N-2}}+\sum_{i=3}^{k-1}\frac{1}{(\beta_k^i)^{N-2}}\right)
+ O(\epsilon^{2\tau}).
\end{eqnarray*}
where we set
\[\beta_k:=|Q_2-Q_{1}|=|Q_k-Q_{1}|=2\sin{\frac{\pi}{k}}\]
and
\[\beta_k^i:=|Q_i-Q_{1}|=\sqrt{2}\sqrt{1-\cos{\frac{2\pi(i-1)}{k}}},\ i=3,...,k-1.\]
Observe that by our choice
\[\beta_k<\beta_k^i,\ i=3,\ldots,k-1\]
hence, from \eqref{dec}, it follows that
\[w\left(\frac{r}{\epsilon}\beta_k^i\right)=o\left(w\left(\frac{r}{\epsilon}\beta_k\right)\right),\ i=3,\ldots,k-1,\ \ \mbox{ as }\frac{r}{\epsilon}\rightarrow +\infty.\]
Moreover $\beta_k<1$ because $k\geq 7,$ hence we also have
\[w\left(\frac{r}{\epsilon}\right)=o\left(w\left(\frac{r}{\epsilon}\beta_k\right)\right),\ \ \mbox{ as }\frac{r}{\epsilon}\rightarrow +\infty.\]
As a consequence the reduced functional becomes
\[
\widetilde M_{\epsilon}[r]=-(2k \gamma_0+o(1))w\left(\frac{r}{\epsilon}\beta_k \right)+\epsilon^2 \left(2kC_k+o(1)\right)     \frac{1}{\left|\frac{r}{\epsilon}\beta_k\right|^{N-2}} + O(\epsilon^{2\tau}),
\]
where ${C_k}:=C_2\left(1+\beta_{k}^{N-2}+\frac{1}{2}\sum_{i=3}^{k-1}\frac{\beta_{k}^{N-2}}{(\beta_k^i)^{N-2}}\right).$
\end{proof}

Finally next result gives an interior maximum point $r$ for $\widetilde M_{\epsilon}$
\begin{proposition}\label{proofMax} For $\epsilon >0$ sufficiently small, the following maximization problem
$$\max\{\widetilde M_{\epsilon}[r]\ : \ r \in \bar R_{\epsilon}\}$$
has a solution ${r}_{\epsilon}\in R_{\epsilon}.$
Furthermore $$\lim_{\epsilon\rightarrow 0}\frac{r_{\epsilon}\beta_k}{\epsilon\log{\frac{1}{\epsilon^2}}}=1.$$
\end{proposition}
\begin{proof}
Since $\widetilde M_{\epsilon}$ is continuous in $r,$ there exists $r_{\epsilon}\in \bar{R_{\epsilon}}$ such that
\[\widetilde M_{\epsilon}[r_{\epsilon}]=\max_{r\in \bar{R}_{\epsilon}}\widetilde M_{\epsilon}[r].\]
We claim that $r_{\epsilon}\in R_{\epsilon}.$ We prove this by energy comparison.
We first obtain a lower bound for $\widetilde M_{\epsilon}[r_{\epsilon}].$
Let us choose $s_{\epsilon}:=\frac{\epsilon \rho_{\epsilon}}{\beta_k}, $ where $\rho_{\epsilon}>0$ is given in Lemma \ref{Lemma Massimo Ridotto}. It is easy to see that $s_{\epsilon}$ belongs to $R_{\epsilon}$. Indeed, by Lemma, \ref{Lemma Massimo Ridotto} $\rho_{\epsilon}>\beta^2\log{\frac{1}{\epsilon^2}}$ and, for $\epsilon$ small,  $\rho_{\epsilon}<\frac{1}{2}\beta_k\left( \log{\frac{1}{\epsilon^2}}\right)^2.$
Then by using again Lemma \ref{Lemma Massimo Ridotto} and \eqref{ridottoSimmetricoPoligono}
\begin{equation}\label{lower bound1}
\widetilde M_{\epsilon}[r_{\epsilon}]\geq \widetilde M_{\epsilon}[s_{\epsilon}] =
(2k+o(1)) \alpha_{\epsilon, C_k}(\rho_{\epsilon})+ O(\epsilon^{2\tau})= \epsilon^2(2kC_k+o(1))  \frac{1}{\left(\log{\frac{1}{\epsilon^2}}\right)^{N-2}}.
\end{equation}
We are going to prove that $\frac{r_{\epsilon}\beta_k}{\epsilon\log{\frac{1}{\epsilon^2}}}\rightarrow 1$ as $\epsilon\rightarrow 0.$
By contradiction assume that there exists a sequence $\epsilon_{n}\rightarrow 0$ such that $\frac{r_{\epsilon_{n}}\beta_k}{\epsilon_{n}\log{\frac{1}{\epsilon_{n}^2}}}>1+c.$ Using once more \eqref{ridottoSimmetricoPoligono}
\begin{equation*}
\widetilde M_{\epsilon_n}[r_{\epsilon_n}]\leq \epsilon_{n}^2(2kC_k+o(1)) \frac{1}{\left|\frac{r_{\epsilon_{n}}}{\epsilon_{n}}\beta_k\right|^{N-2}} +O(\epsilon_{n}^{2\tau})\leq
\epsilon_{n}^2(2kC_k+o(1)) \frac{1}{(1+c)^{N-2}} \frac{1}{\left|\log{\frac{1}{\epsilon_{n}^2}}\right|^{N-2}}
\end{equation*}
which contradicts the \eqref{lower bound1}.
Now assume the existence of a sequence  $\epsilon_{n}\rightarrow 0$ such that $\frac{r_{\epsilon}\beta_k}{\epsilon\log{\frac{1}{\epsilon^2}}}<1-c.$ Then by the decay of $w$ (observe that the function $x\mapsto \frac{e^x}{x^{\frac{N-3}{2}}}$ is nondecreasing for $x$ large)
\begin{eqnarray*}\widetilde M_{\epsilon_{n}}[r_{\epsilon_n}]&=& (2k+o(1))\left(-\gamma_0w\left(\frac{r_{\epsilon_n}}{\epsilon_n}\beta_k\right)+C_k\epsilon_n^2\frac{1}{\left|\frac{r_{\epsilon_n}}{\epsilon_n}\beta_k\right|^{N-2}}  \right)+ O(\epsilon_n^{2\tau})
\\
&&= (2k+o(1))\left( -\gamma_0A_N\frac{e^{-\left(\frac{r_{\epsilon_n}}{\epsilon_n}\beta_k\right)}}{\left| \frac{r_{\epsilon_n}}{\epsilon_n}\beta_k \right|^{\frac{N-1}{2}}} +C_k\epsilon_n^2\frac{1}{\left|\frac{r_{\epsilon_n}}{\epsilon_n}\beta_k\right|^{N-2}} \right) + O(\epsilon_n^{2\tau})\\
&&\leq (2k+o(1))\left(  -\gamma_0A_N\frac{e^{-\left(\frac{r_{\epsilon_n}}{\epsilon_n}\beta_k\right)}}{\left| \frac{r_{\epsilon_n}}{\epsilon_n}\beta_k \right|^{N-2}} +C_k\epsilon_n^2\frac{1}{\left|\frac{r_{\epsilon_n}}{\epsilon_n}\beta_k\right|^{N-2}} \right) + O(\epsilon_n^{2\tau})\\
&&=\epsilon_n^2 (2kC_k+o(1))
\left(  -\frac{\gamma_0A_N}{\epsilon_n^2C_k}\frac{e^{-\left(\frac{r_{\epsilon_n}}{\epsilon_n}\beta_k\right)}}{\left| \frac{r_{\epsilon_n}}{\epsilon_n}\beta_k \right|^{N-2}} +\frac{1}{\left|\frac{r_{\epsilon_n}}{\epsilon_n}\beta_k\right|^{N-2}} \right) + O(\epsilon_n^{2\tau})\\
&&\leq \epsilon_n^2 (2kC_k+o(1))\frac{1}{\left|\log{\frac{1}{\epsilon_n^2}}\right|^{N-2}}
\left(  -\frac{\gamma_0A_N}{C_k}\frac{\epsilon_n^{-2c}}{\left| (1-c) \right|^{N-2}} +\frac{1}{\left|\beta^2\right|^{N-2}} \right)
\end{eqnarray*}
which is in contradiction with \eqref{lower bound1} since $\left(  -\frac{\gamma_0A_N}{C_k}\frac{\epsilon_n^{-2c}}{\left| (1-c) \right|^{N-2}} +\frac{1}{\left|\beta^2\right|^{N-2}} \right)<1$ for $n$ big.
\end{proof}

\

\

\section{Proof of Theorem \ref{theorem politopo}}
For every $x\in\mathbb R^N$ and $h\in [2,N]$ we set $x=(x_1,\ldots,x_N)=(z,x'),$ where $z:=(x_1,\ldots, x_h)\in \mathbb R^h$ and $x':=(x_{h+1},\ldots, x_N)\in \mathbb R^{N-h}.$ \\
\\ In this section we prove the existence of a cluster solution to \eqref{eqScalare} having a positive bump in $0$ and $k$ negative bumps at the vertices of a regular polytope $\mathcal P$ centered  in $0$. \\
The proof is similar to the one of Theorem \ref{theorem poligono piano} (which is actually  a special case of Theorem \ref{theorem politopo} with $h=2$). For this reason we will only sketch it briefly, emphasizing the main differences.\\
\\
Let $Q_i:=(z_i,{\bf 0}), i=1,\ldots,k$  be the vertices of a convex regular polytope $\mathcal P$ in $\mathbb R^h$ centered in the origin and having radius $1$ and side $s,$ with $s\leq 1$ when $h>2$, $s<1$ when $h=2$,  we  make now the following
{\it{ansatz}}
\[v=w_0-\sum_{i=1}^kw_{P_i}+\psi_{\epsilon},\]
where
$P_i:=rQ_i\in\mathbb R^N,$   and
$r\in R_{\epsilon},$ where
\[R_{\epsilon}:=\left\{r >0\ : \ \frac{\beta^2}{s}\epsilon\log{\frac{1}{\epsilon^2}}<r <\frac{1}{2}\epsilon \left(\log{\frac{1}{\epsilon^2}}\right)^2
\right\}.\]
\\
Observe that with this choice of $r$ the point ${\bf P}_r:=(0,P_1,\ldots,P_k)\in \Gamma_{\epsilon}(\subset \mathbb R^{(k+1)N})$, where $\Gamma_{\epsilon}$ is the configuration set
introduced in \eqref{gammeepsilon}. Indeed, by the definition of $\Gamma_{\epsilon}$, ${\bf P}_r\in \Gamma_{\epsilon}$ if and only if
\[\left\{
\begin{array}{lr}
 \beta^2\epsilon\log{\frac{1}{\epsilon^2}}<r <\epsilon \left(\log{\frac{1}{\epsilon^2}}\right)^2\\
\beta^2\epsilon\log{\frac{1}{\epsilon^2}}<r|Q_i-Q_j| <\epsilon \left(\log{\frac{1}{\epsilon^2}}\right)^2\quad i\neq j.
\end{array}
\right.\]
and by construction $s\leq |Q_i-Q_j|\leq 2$ for $i\neq j,$ and by assumption $s\leq 1$.
\\
\\
Generalizing what we have done in the previous section, we may  assume that $\mathcal P$ is invariant by reflection with respect to the  hyperplanes $x_{1}x_{i_2}\dots x_{i_{h-1}},$ where $i_m\in \{2,\dots, h\}$.
The Sobolev space of symmetric functions in which to work is now
\[X:=\{v\in H^1(\mathbb R^N): v(z,x')=v(z,|x'|)=v(g z,|x'|)\ \ \forall g\in \mathcal G\},\]
where $\mathcal G$ is the Coxeter group of $\mathbb R^h$ associated to $\mathcal P,$ namely the symmetry group that leaves invariant $\mathcal P$
(and which contains  all the reflections with respect to the  hyperplanes $x_{1}x_{i_2}\dots x_{i_{h-1}},$ for $i_m\in\{2, \dots, h\}$).
Observe that $X=\{u\in H^1(\mathbb R^N): u(T_{g,h}(x))=u(x), \ \mbox{ for all }\  g\in \mathcal G, h\in\mathcal R\},$ where
\[T_{g,h}:\mathbb R^N\rightarrow\mathbb R, \quad T_{g,h}(y)=T_{g,h}(z,y')=(gz,hy'),\]
and $\mathcal R$ is the group of the rotation matrix in $\mathbb R^{N-h}.$
 Moreover, by Lemma \ref{invact} the functional $J_{\epsilon}$ is invariant under the action of the group $\{T_{g,h}:\ g \in \mathcal G, \, h\in\mathcal R\},$ namely $J_{\epsilon}(u(T_{g,h}(x)))=J_{\epsilon}(u(x)).$\\

The analogous of Lemma \ref{LemmaAusiliariaSimmetrico1} is now the following (we omit the proof)
\begin{lemma}\label{LemmaAusiliariaSimmetrico2}
Fix $\tau=\beta^4(1+\sigma)$. Provided $\epsilon>0$ is sufficiently small, for every $r\in R_{\epsilon}$ there is a unique pair $(\psi_{\bf P_r}, c_{1, 1}({\bf P_r}))\in H^2_{*,r,s}\times \mathbb R$ which solve 
\begin{equation} \label{ausiliariaRidottaPersimmetria}
\left\{
\begin{array}{lr}
S_{\epsilon}[w_{\bf P_r}+\psi]=c_{1, 1}\sum_{\substack{i=1,\dots, k\\j=1,\dots,h}} \alpha_{i,j} Z_{P_{i}, j},\\
\psi\in H^2_{*, r,s},\,\, \langle \psi, Z_{{P}_{i}, j}\rangle=0,\,\,\ i=1, \ldots, \ell,\,\, j=1, \ldots, N.
\end{array}
\right.,
\end{equation}
where $\alpha_{i,j}$ are known numbers.
Moreover
\[
\|\psi_{\bf P_r}\|_{*,{\bf P}_r}<\epsilon^{\tau}; \quad (\psi_{\bf P_r}, \psi_{\bf P_r})_{\epsilon}\leq \epsilon^{N+2\tau}
\]
the map ${r}\in \bar{R}_{\epsilon}\longmapsto \psi_{\bf P_r}\in H^1(\mathbb R^N)$ is $C^1$ and the map ${r}\longmapsto c_{1, 1}(\bf P_r)\in \mathbb R$ is continuous.
\end{lemma}
$\;$\\
We set, for $\epsilon>0$ sufficiently small, the one variable function
$$\widetilde M_{\epsilon}[r]:=M_{\epsilon}[{\bf P}_r]$$
where $M_{\epsilon}$ is the reduced functional defined in \eqref{funzioRidot}, and, similarly as in Section \ref{theorem poligono piano} we can prove Lemma \ref{lemmarelazione punti critici ridotto e funzionale}.
Hence in order to conclude the proof we need to find a critical point for the reduced functional.\\

The reduced functional reduces  to the following
\begin{proposition}
For $\epsilon>0$ sufficiently small, if $s<1$
\begin{equation}\label{ridottoSimmetricoPolitopo1}
\widetilde M_{\epsilon}[r]= (qk+o(1)) \alpha_{\epsilon, C_{\mathcal P}}\left(\frac{r}{\epsilon}s\right)+ O(\epsilon^{2\tau}),
\end{equation}
if $s=1$ and $q\neq 2,$
\begin{equation}\label{ridottoSimmetricoPolitopo2}
\widetilde M_{\epsilon}[r]= ((q-2)k+o(1)) \alpha_{\epsilon, C_{\mathcal P}'}\left(\frac{r}{\epsilon}\right)+ O(\epsilon^{2\tau}),
\end{equation}
uniformly for  $r>0$ such that ${\bf P}_r\in \bar{\Gamma}_{\epsilon},$ where $q$ denotes the number of vertices $Q_i$ which are one side away from $Q_1,$ $C_{\mathcal P}, C_{\mathcal P}' $ are positive constants and $\alpha_{\epsilon, C_{\mathcal P}}$ is the function defined in Lemma \ref{Lemma Massimo Ridotto}.
\end{proposition}
\begin{remark}Under the assumptions of Theorem \ref{theorem politopo} $q\neq 2$ when $s=1,$ indeed $q\geq h>2.$
\end{remark}
\begin{proof}
For ${r}$ such that ${\bf P}_r\in \bar{\Gamma}_{\epsilon}$ the reduced functional becomes
\begin{eqnarray*}
\widetilde M_{\epsilon}[r]&=& - (\gamma_0+o(1))\left( -2k w\left(\frac{r}{\epsilon}\right)+
\sum_{i\neq j}w\left(\frac{r}{\epsilon}|Q_i-Q_j|\right)\right) \\
&&+\epsilon^{2}(C_2+o(1))\frac{1}{\left|\frac{r}{\epsilon}\right|^{N-2}}\left(2k+\sum_{i\neq j}\frac{1}{|Q_i-Q_j|^{N-2}}\right)
+ O(\epsilon^{2\tau})\\
&=& - k(\gamma_0+o(1))\left( -2 w\left(\frac{r}{\epsilon}\right)+
\sum_{i=2}^kw\left(\frac{r}{\epsilon}|Q_1-Q_i|\right)\right)\\
&&+\epsilon^{2}k(C_2+o(1))\frac{1}{\left|\frac{r}{\epsilon}\right|^{N-2}}\left(2+\sum_{i=2}^k\frac{1}{|Q_1-Q_i|^{N-2}}\right)
+ O(\epsilon^{2\tau})\\
& =&
 - k(\gamma_0+o(1))\left( -2 w\left(\frac{r}{\epsilon}\right)+ q w\left(\frac{r}{\epsilon}s\right)+
\sum_{\{i: s_i>s\}}w\left(\frac{r}{\epsilon}s_i\right)\right)\\
&&+\epsilon^{2}k(C_2+o(1))\frac{1}{\left|\frac{r}{\epsilon}\right|^{N-2}}\left(2+q\frac{1}{s^{N-2}}+\sum_{\{i: s_i>s\}}\frac{1}{s_i^{N-2}}\right)
+ O(\epsilon^{2\tau})
\end{eqnarray*}
where we set
\[s_i:=|Q_i-Q_{1}|.\]
From the exponential decay of $w$ it follows that for  $s_i>s$
\[w\left(\frac{r}{\epsilon}s_i\right)=o\left(w\left(\frac{r}{\epsilon}s\right)\right),\ \ \ \mbox{ as }\frac{r}{\epsilon}\rightarrow +\infty,\]
hence
\[\widetilde M_{\epsilon}[r]= - k(\gamma_0+o(1))\left( -2 w\left(\frac{r}{\epsilon}\right)+ q w\left(\frac{r}{\epsilon}s\right)\right) +\epsilon^{2}k(C_2+o(1))\frac{1}{\left|\frac{r}{\epsilon}\right|^{N-2}}\left(2+q\frac{1}{s^{N-2}}+\sum_{\{i: s_i>s\}}\frac{1}{s_i^{N-2}}\right)
+ O(\epsilon^{2\tau}).\]
If $s<1$ we have
\[w\left(\frac{r}{\epsilon}\right)=o\left(w\left(\frac{r}{\epsilon}s\right)\right),\ \ \mbox{ as }\frac{r}{\epsilon}\rightarrow +\infty.\]
As a consequence the reduced functional becomes
\[
\widetilde M_{\epsilon}[r]=-(qk \gamma_0+o(1))w\left(\frac{r}{\epsilon}s \right)+\epsilon^2 \left(qkC_{\mathcal P}+o(1)\right)     \frac{1}{\left|\frac{r}{\epsilon}s\right|^{N-2}} + O(\epsilon^{2\tau}),
\]
where ${C_{\mathcal P}}:=C_2\left(1+\frac{2}{q}s^{N-2}+\frac{1}{q}\sum_{\{i: s_i>s\}}\frac{s^{N-2}}{s_i^{N-2}}\right).$
\\
\\
While if $s=1$ then
\[
\widetilde M_{\epsilon}[r]=-((q-2)k \gamma_0+o(1))w\left(\frac{r}{\epsilon} \right)+\epsilon^2 \left(qkC_{\mathcal P}+o(1)\right)     \frac{1}{\left|\frac{r}{\epsilon}\right|^{N-2}} + O(\epsilon^{2\tau}).\]
\end{proof}

The following result (which can be proved similarly as Proposition \ref{proofMax}) concludes the proof
\begin{proposition}Assume $s\leq 1,$ $h>2$ or $s<1$ $h=2.$ For $\epsilon >0$ sufficiently small, the following maximization problem
\[\max\{\widetilde M_{\epsilon}[r]\ : \ r \in \bar R_{\epsilon}\}\]
has a solution ${r}_{\epsilon}\in R_{\epsilon}.$
Furthermore \[\lim_{\epsilon\rightarrow 0}\frac{r_{\epsilon}s}{\epsilon\log{\frac{1}{\epsilon^2}}}=1.\]
\end{proposition}

\

\section*{Appendix}

\begin{proof}[Proof of Lemma \ref{LemmaAusiliariaSimmetrico1}]
Let $\bf P=\bf P_r$ be as in \eqref{ansatzSimmetrico1} and $H^2_{\ast,r,s}$ be the space of symmetric functions defined in \eqref{spazioSimmetrico1}.
\\
Then proceeding as in Sects. \ref{section linearized} and \ref{Section finite reduction}, we obtain the analogous of see Lemma \ref{fdr} in the symmetric case, namely  for $\epsilon$ small enough we find a unique solution $(\psi_{{\bf P}_r}, c_{ij}({\bf P}_r))\in H^2_{\ast, r, s}\times \mathbb R^{N(k+1)}$ to problem 
\begin{equation}\label{AusEq}
\left\{
\begin{array}{lr}
S_{\epsilon}[w_{\bf P}+\psi]=\sum_{i, j}c_{i, j}Z_{P_{i}, j},\\
\psi\in H^2_{*, {\bf P}},\,\, \langle \psi, Z_{P_{i}, j}\rangle=0,\,\,\ i=1, \ldots, k+1,\,\, j=1, \ldots, N.
\end{array}
\right.
\end{equation}
Following \cite{DelPinoFelmerMusso} we show that the right hand side in \eqref{AusEq} reduces because of the symmetries of $\psi_{{\bf P}}$ and the symmetries in the choice of $\bf P$.\\\\
Let us first analyze these symmetries. Observe that 
 $w_{{\bf P}},\psi_{{\bf P}}$ are even with respect to $x_h$, $h=2,\dots,N$, namely
\begin{equation}\label{simmetriaRispettoxh} w_{{\bf P}}(\dots, x_h,\dots)=w_{{\bf P}}(\dots, -x_h,\dots), \  \ \psi_{{\bf P}}(\dots, x_h,\dots)=\psi_{{\bf P}}(\dots, -x_h,\dots), \ \ h=2,\dots,N.
\end{equation}
Moreover $w_{{\bf P}},\psi_{{\bf P}}$ are invariant by the following rotation: 
\begin{equation}\label{simmetriaRispRota}
w_{{\bf P}}(z, x')=w_{{\bf P}}(ze^{2\pi\sqrt{-1}/k}, x'), \qquad \ \psi_{{\bf P}}(z, x')=\psi_{{\bf P}}(ze^{2\pi\sqrt{-1}/k}, x')
\end{equation}
 \\
 \\
 Moreover each $w_{{P_i}}$, $i=1,\dots,k+1$ is even with respect to $x_h$, $h=3,\dots,N$, hence an easy computation shows that
%
%
%
%
\begin{equation}\label{simmPihmagg3}
\left\{\begin{array}{lr}
\mbox{ for } h=3,\dots, N, \  i=1,\dots, k+1\\
Z_{P_i,j}(\dots, x_h,\dots)=\left\{\begin{array}{ll} -Z_{P_i,j}(\dots, -x_h,\dots)&\ \mbox{ if }j=h\\Z_{P_i,j}(\dots, -x_h,\dots)&\ \mbox{ if }j=1\dots, N; \ j\neq h \end{array} \right.,
\end{array}\right.
\end{equation}  
while, for $h=2$, only $w_{P_1}$ and $w_{P_{k+1}}$ are even with respect to $x_2$, and so 
\begin{equation}\label{simmx2prima}
\left\{\begin{array}{lr}
i\in\{1,k+1\}\\
Z_{P_i,j}(x_1, x_2,x')=\left\{\begin{array}{ll} -Z_{P_i,j}(x_1, -x_2,x')&\ \mbox{ if }j=2\\Z_{P_i,j}(x_1, -x_2,x')&\ \mbox{ if } \ j=1\dots, N;\ j\neq 2\end{array}\right..
\end{array}\right. 
\end{equation} 
While, for $i=2,\dots, k$, clearly $w_{{P_i}}$ is not even with respect to $x_2$, anyway,  by the choice of the configuration $\bf P$, the point $P_{k+2-i}$ turns out to be the symmetric of $P_i$ through the reflection with respect to the $x_1$-axis, and so it is not difficult to see that  
\begin{equation}\label{simmx2seconda}
\left\{\begin{array}{lr} i=2,\dots, k\\
Z_{P_i,j}(x_1, x_2,x')=\left\{\begin{array}{ll} -Z_{P_{k+2-i},j}(x_1, -x_2,x')&\ \mbox{ if }j=2\\
Z_{P_{k+2-i},j}(x_1, -x_2,x')&\ \mbox{ if } \ j=1\dots, N ;\ j\neq 2 \end{array}  \right.
\end{array}\right.
\end{equation}  
$\;$
\\
About the rotation, let us observe that, because of the symmetry in the choice of $\bf P$:
\begin{equation}\label{simmZRota} 
\begin{array}{lr}
Z_{P_i,1}(ze^{2\pi\sqrt{-1}/k}, x')=\cos{\frac{2\pi}{k}} Z_{P_{i-1},1}(z, x')- \sin{\frac{2\pi}{k}} Z_{P_{i-1},2}(z, x')\\
Z_{P_i,2}(ze^{2\pi\sqrt{-1}/k}, x')=\sin{\frac{2\pi}{k}} Z_{P_{i-1},1}(z, x')+ \cos{\frac{2\pi}{k}} Z_{P_{i-1},2}(z, x')
\end{array}\ \mbox{ for } i=1,\dots, k
\end{equation}
with the convention that $P_0=P_k$, while 
\begin{equation}\label{simmZRota2} 
\begin{array}{lr}Z_{P_{k+1},1}(ze^{2\pi\sqrt{-1}/k}, x')=\cos{\frac{2\pi}{k}} Z_{P_{k+1},1}(z, x')- \sin{\frac{2\pi}{k}} Z_{P_{k+1},2}(z, x')\\
Z_{P_{k+1},2}(ze^{2\pi\sqrt{-1}/k}, x')=\sin{\frac{2\pi}{k}} Z_{P_{k+1},1}(z, x')+ \cos{\frac{2\pi}{k}} Z_{P_{k+1},2}(z, x')
\end{array}.
\end{equation}
\\
 \\
We are now ready to prove the result. By \eqref{AusEq}, \eqref{simmetriaRispettoxh} and \eqref{simmPihmagg3} it follows that, for $h=3,\dots, N$
\begin{eqnarray*}\sum_{\substack{i=1,\dots, k+1\\j=1,\dots,N}}c_{i, j}Z_{P_{i}, j}(\dots, x_h,\dots)
& \stackrel{\eqref{AusEq}, \eqref{simmetriaRispettoxh}}{ =} & \sum_{\substack{i=1,\dots, k+1\\j=1,\dots,N}}c_{i, j}Z_{P_{i}, j}(\dots, -x_h,\dots)\\
&\stackrel{\eqref{simmPihmagg3}}{=}& \sum_{\substack{i=1,\dots, k+1\\j=1,\dots,N\\ j\neq h}}c_{i, j}Z_{P_{i}, j}(\dots, x_h,\dots)-\sum_{\substack{i=1,\dots, k+1}}c_{i, h}Z_{P_{i}, h}(\dots, x_h,\dots),
\end{eqnarray*} 
namely
\[\sum_{\substack{i=1,\dots, k+1}}c_{i, h}Z_{P_{i}, h}=0,\quad \mbox{ for } h=3,\dots, N.\]
Multiplying by $\frac{\partial w_{P_m}}{\partial x_h}$, $m=1,\dots, k+1$, $h=3,\dots, N$, integrating and using \eqref{prodotto} we then obtain that
\[c_{i, h}=0, \quad \mbox{ for }  i=1,\dots, k+1; \  \ h=3,\dots, N.\]
So  the right hand side in \eqref{AusEq} reduces to
\begin{equation}\label{primariduzione}
S_{\epsilon}[w_{\bf P}+\psi]=\sum_{\substack{i=1,\dots, k+1\\j=1,\dots,N}}c_{i, j}Z_{P_{i}, j}=\sum_{\substack{i=1,\dots, k+1}}
\left(c_{i, 1}Z_{P_{i}, 1}+c_{i, 2}Z_{P_{i}, 2}\right).
\end{equation}
By  \eqref{primariduzione}, \eqref{simmetriaRispRota} and \eqref{simmZRota}, \eqref{simmZRota2}  we then have (recall that by our convention $P_0=P_k$)
\begin{eqnarray*}
&&\sum_{\substack{i=1,\dots, k+1}}
\left(c_{i, 1}Z_{P_{i}, 1}(z,x')+c_{i, 2}Z_{P_{i}, 2}(z,x')\right)
\\
&&\qquad \qquad \stackrel{\eqref{primariduzione}, \eqref{simmetriaRispRota}}{=}\sum_{\substack{i=1,\dots, k+1}}
\left(c_{i, 1}Z_{P_{i}, 1}(ze^{2\pi\sqrt{-1}/k},x')+c_{i, 2}Z_{P_{i}, 2}(ze^{2\pi\sqrt{-1}/k},x')\right)
\\
&&\qquad \qquad \stackrel{\eqref{simmZRota}, \eqref{simmZRota2}}{=}\sum_{\substack{i=1,\dots, k}} 
\left(
c_{i, 1} \cos{\frac{2\pi}{k}} + c_{i, 2} \sin{\frac{2\pi}{k}}
\right)
Z_{P_{i-1}, 1}\ + \
\left(-c_{i, 1} \sin{\frac{2\pi}{k}} + c_{i, 2} \cos{\frac{2\pi}{k}}\right)
Z_{P_{i-1}, 2}\\
&&\qquad\qquad\qquad +\quad\qquad  \left(
c_{k+1, 1} \cos{\frac{2\pi}{k}} + c_{k+1, 2} \sin{\frac{2\pi}{k}}
\right)
Z_{P_{k+1}, 1}\ + \
\left(-c_{k+1, 1} \sin{\frac{2\pi}{k}} + c_{k+1, 2} \cos{\frac{2\pi}{k}}\right)
Z_{P_{k+1}, 2}.
\end{eqnarray*}
Namely
\begin{eqnarray*}
&& 
\sum_{\substack{i=1,\dots, k}}\left[c_{i-1,1}-  \left(c_{i, 1} \cos{\frac{2\pi}{k}} + c_{i, 2} \sin{\frac{2\pi}{k}}
\right)\right]
Z_{P_{i-1}, 1}\ +\ \left[c_{i-1,2}-\left(-c_{i, 1} \sin{\frac{2\pi}{k}} + c_{i, 2} \cos{\frac{2\pi}{k}}\right)\right] Z_{P_{i-1}, 2}
\\
&&\qquad + \ \left[c_{k+1,1}-  \left(c_{k+1, 1} \cos{\frac{2\pi}{k}} + c_{k+1, 2} \sin{\frac{2\pi}{k}}
\right)\right]
Z_{P_{k+1}, 1}
\\
&&\qquad +\ \left[c_{k+1,2}-\left(-c_{k+1, 1} \sin{\frac{2\pi}{k}} + c_{k+1, 2} \cos{\frac{2\pi}{k}}\right)\right] Z_{P_{k+1}, 2}
\\
&&\ = \ 0.
\end{eqnarray*}
Multiplying by $\frac{\partial w_{P_m}}{\partial x_h}$, $m=1,\dots, k+1$, $h=1,2$, integrating and using \eqref{prodotto} we then obtain that
\[
\left\{
\begin{array}{lr}
c_{i-1,1} = c_{i, 1} \cos{\frac{2\pi}{k}} + c_{i, 2} \sin{\frac{2\pi}{k}}
\\
c_{i-1,2} = -c_{i, 1} \sin{\frac{2\pi}{k}} + c_{i, 2} \cos{\frac{2\pi}{k}}
\end{array}
\right., \ \ i=1,\dots, k
\]
and
\[
\left\{
\begin{array}{lr}
c_{k+1,1} = c_{k+1, 1} \cos{\frac{2\pi}{k}} + c_{k+1, 2} \sin{\frac{2\pi}{k}}
\\
c_{k+1,2} = -c_{k+1, 1} \sin{\frac{2\pi}{k}} + c_{k+1, 2} \cos{\frac{2\pi}{k}}
\end{array}
\right.
\]
From this it follows easily that
\[
\left\{
\begin{array}{lr}
c_{i,1} = c_{1, 1} \cos{\frac{2\pi (i-1)}{k}} - c_{1, 2} \sin{\frac{2\pi(i-1)}{k}}
\\
c_{i,2} = c_{1, 1} \sin{\frac{2\pi(i-1)}{k}} + c_{1, 2} \cos{\frac{2\pi(i-1)}{k}}
\end{array}
\right., \ \ i=1,\dots, k
\]
and 
\[
\left\{
\begin{array}{lr}
c_{k+1,1} =0
\\
c_{k+1,2} = 0
\end{array}
\right.
\]
So, from  \eqref{primariduzione}, the right hand side in \eqref{AusEq} reduces again to
\begin{equation}\label{secondariduzione}
S_{\epsilon}[w_{\bf P}+\psi]=\sum_{\substack{i=1,\dots, k}}
\left[\left(c_{1, 1} \cos{\frac{2\pi (i-1)}{k}} - c_{1, 2} \sin{\frac{2\pi(i-1)}{k}}\right)Z_{P_{i}, 1}+\left(c_{1, 1} \sin{\frac{2\pi(i-1)}{k}} + c_{1, 2} \cos{\frac{2\pi(i-1)}{k}} \right)Z_{P_{i}, 2}\right].
\end{equation}
Last we use the symmetry with respect to the $x_2$-variable. By \eqref{secondariduzione}, \eqref{simmx2prima} and \eqref{simmx2seconda}
\begin{eqnarray*}
S_{\epsilon}[w_{\bf P}(x_1,-x_2,x')+\psi(x_1,-x_2,x')] &= & c_{1,1} Z_{P_{1}, 1}(x_1,x_2,x') -c_{1,2}Z_{P_{1}, 2}(x_1,x_2,x')
\\ 
&+& \sum_{\substack{i=2,\dots, k}}
\left[\left(c_{1, 1} \cos{\frac{2\pi (i-1)}{k}} - c_{1, 2} \sin{\frac{2\pi(i-1)}{k}}\right)Z_{P_{k+2-i}, 1}(x_1,x_2,x')\ + \right.\\
&&\qquad\qquad\left.
-\ \left(c_{1, 1} \sin{\frac{2\pi(i-1)}{k}} + c_{1, 2} \cos{\frac{2\pi(i-1)}{k}} \right)Z_{P_{k+2-i}, 2}(x_1,x_2,x')\right]
\\ 
&= &  c_{1,1} Z_{P_{1}, 1}(x_1,x_2,x') -c_{1,2}Z_{P_{1}, 2}(x_1,x_2,x')
\\
&+& \sum_{\substack{i=2,\dots, k}}
\left[\left(c_{1, 1} \cos{\frac{2\pi (k+1-i)}{k}} - c_{1, 2} \sin{\frac{2\pi(k+1-i)}{k}}\right)Z_{P_{i}, 1}(x_1,x_2,x')\ + \right.\\
&&\qquad\quad\left.
-\ \left(c_{1, 1} \sin{\frac{2\pi(k+1-i)}{k}} + c_{1, 2} \cos{\frac{2\pi(k+1-i)}{k}} \right)Z_{P_{i}, 2}(x_1,x_2,x')\right] \\
&= & \sum_{\substack{i=1,\dots, k}}
\left[\left(c_{1, 1} \cos{\frac{2\pi (i-1)}{k}} + c_{1, 2} \sin{\frac{2\pi (i-1)}{k}}\right)Z_{P_{i}, 1}(x_1,x_2,x')\ + \right.\\
&&\qquad\quad\left.
-\ \left(-c_{1, 1} \sin{\frac{2\pi (i-1)}{k}} + c_{1, 2} \cos{\frac{2\pi (i-1)}{k}} \right)Z_{P_{i}, 2}(x_1,x_2,x')\right] 
\end{eqnarray*}
where in the last equality we used that the angles $-\frac{2\pi (i-1)}{k}=\frac{2\pi(k+1-i)}{k}$. 
And so,  by the symmetry with respect to the $x_2$-variable:
\[S_{\epsilon}[w_{\bf P}(x_1,x_2,x')+\psi(x_1,x_2,x')] = S_{\epsilon}[w_{\bf P}(x_1,-x_2,x')+\psi(x_1,-x_2,x')],\]
namely
\begin{eqnarray*}
&&\sum_{\substack{i=1,\dots, k}}
\left[ c_{1, 2} \sin{\frac{2\pi (i-1)}{k}}Z_{P_{i}, 1}(x_1,x_2,x')\ 
 - c_{1, 2} \cos{\frac{2\pi (i-1)}{k}} Z_{P_{i}, 2}(x_1,x_2,x')\right] =0
\end{eqnarray*}
From which, multiplying by $\frac{\partial w_{P_1}}{\partial x_2 }$, integrating and using \eqref{prodotto} we get
\[c_{1,2}=0.\]
As a consequence from  \eqref{secondariduzione}, the right hand side in \eqref{AusEq} reduces again, and we obtain
\begin{equation}\label{terzariduzione}
S_{\epsilon}[w_{\bf P}+\psi]= 
c_{1, 1}
\sum_{\substack{i=1,\dots, k}}
\left( \cos{\frac{2\pi (i-1)}{k}} Z_{P_{i}, 1}\ + \ \sin{\frac{2\pi(i-1)}{k}} Z_{P_{i}, 2}\right).
\end{equation}
\end{proof}
%
%
%


\

\end{document}